\documentclass[a4paper,UKenglish,cleveref, autoref]{lipics-v2019}

\newif{\ifdraft}\drafttrue
\draftfalse

\newif{\ifdrawFigures}\drawFigurestrue

\newif{\ifFullVersion}\FullVersiontrue

\usepackage{todonotes}

\usepackage{mathtools}
\usepackage[utf8]{inputenc}
\usepackage{amssymb,amsmath,enumerate}
\usepackage{fancyhdr}
\usepackage{stmaryrd}
\usepackage{xspace}
\usepackage{tikz}
\usetikzlibrary{positioning}
\usetikzlibrary{decorations.pathreplacing}
\usetikzlibrary{arrows.meta}
\usepackage{multicol}
 \usepackage{amsthm}
 \usepackage{array}

\makeatletter
\newcommand{\thickhline}{%
    \noalign {\ifnum 0=`}\fi \hrule height 1pt
    \futurelet \reserved@a \@xhline
}
\newcolumntype{"}{@{\hskip\tabcolsep\vrule width 1pt\hskip\tabcolsep}}
\makeatother
\newtheorem{lemma*}{Lemma}

\theoremstyle{definition}

\usepackage{amsfonts}
\usepackage{amssymb}
\usepackage{amsmath}

\usepackage{tabularx,hhline,rotating,xspace}
\usepackage{tikz}
\usetikzlibrary{calc}

\usepackage{prettyref}

\renewcommand{\setminus}{\mysetminus}

\newcommand{\red}[1]{\operatorname{red}(#1)}

\newenvironment{aw}{\noindent\color{magenta} AW : }{}

\makeindex

 \newcommand{\prref}\prettyref
\newrefformat{thm}{Theorem~\ref{#1}}
\newrefformat{lem}{Lem\-ma~\ref{#1}}
\newrefformat{def}{Definition~\ref{#1}}
\newrefformat{cor}{Corollary~\ref{#1}} 
\newrefformat{prop}{Proposition~\ref{#1}}
\newrefformat{sec}{Section~\ref{#1}}
\newrefformat{cap}{Chapter~\ref{#1}}
\newrefformat{bsp}{Example~\ref{#1}}
\newrefformat{rem}{Remark~\ref{#1}}
\newrefformat{fig}{Figure~\ref{#1}}
\newrefformat{eq}{(\ref{#1})}
\newrefformat{ex}{Example~\ref{#1}}
\newrefformat{tab}{Table~\ref{#1}}
\newrefformat{app}{Appendix~\ref{#1}}
\newrefformat{alg}{Algorithm~\ref{#1}}

\newcommand{\mysetminusD}{\raisebox{.8pt}{\hbox{\tikz{\draw[line width=0.6pt,line cap=round] (3.5pt,0pt) -- (0,5.2pt);}}}}
\newcommand{\mysetminusT}{\mysetminusD}
\newcommand{\mysetminusS}{\raisebox{.5pt}{\hbox{\tikz{\draw[line width=0.45pt,line cap=round] (2.2pt,0) -- (0,3.8pt);}}}}
\newcommand{\mysetminusSS}{\raisebox{.35pt}{\hbox{\tikz{\draw[line width=0.4pt,line cap=round] (1.5pt,0) -- (0,2.8pt);}}}}

\newcommand{\mysetminus}{\mathbin{\mathchoice{\mysetminusD}{\mysetminusT}{\mysetminusS}{\mysetminusSS}}}

\newcommand{\supp}{\mathop\mathrm{supp}}

\newcommand{\set}[2]{\left\{\, \mathinner{#1}\vphantom{#2}\: \left|\: \vphantom{#1}\mathinner{#2} \right.\,\right\}}
\newcommand{\oneset}[1]{\left\{\, \mathinner{#1} \,\right\}}

\newcommand{\smallset}[1]{\left\{\mathinner{#1}\right\}}

\newcommand{\abs}[1]{\left|\mathinner{#1}\right|}

\newcommand{\N}{\ensuremath{\mathbb{N}}}
\newcommand{\Z}{\ensuremath{\mathbb{Z}}}

\newcommand{\NP}{\ensuremath{\mathsf{NP}}\xspace} %
\newcommand{\coNP}{\ensuremath{\mathsf{coNP}}\xspace}
\newcommand{\LOGCFL}{\ensuremath{\mathsf{LOGCFL}}\xspace} %
\newcommand{\DLOGTIME}{\ensuremath{\mathsf{DLOGTIME}}\xspace} %
\renewcommand{\L}{\ensuremath{\mathsf{L}}\xspace} %
\newcommand{\TC}{\ensuremath{\mathsf{uTC}^0}\xspace}
\newcommand{\Tc}[1]{\ensuremath{\mathsf{TC}^{#1}}\xspace}
\newcommand{\Ac}[1]{\ensuremath{\mathsf{AC}^{#1}}\xspace}
\newcommand{\Nc}[1]{\ensuremath{\mathsf{NC}^{#1}}\xspace}

\newcommand{\AC}{\ensuremath{\mathsf{uAC}^0}\xspace}
\newcommand{\NC}{\ensuremath{\mathsf{NC}}\xspace}
\renewcommand{\P}{\ensuremath{\mathsf{P}}\xspace}

\newcommand{\sign}{\mathop{\mathrm{sign}}}

\renewcommand{\phi}{\varphi}

\newcommand{\id}{\mathrm{id}}

\newcommand{\Sig}{\Sigma}

\newcommand{\Oh}{\mathcal{O}}

\newcommand{\cC}{\mathcal{C}}

\newcommand{\cS}{\mathcal{S}}

\newcommand\RAS[2]{\overset{#1}{\underset{#2}{\Longrightarrow}}}

\newcommand\DAS[2]{\overset{#1}{\underset{#2}{\Longleftrightarrow}}}

\newcommand\IRR{\mathop\mathrm{IRR}}

\newcommand{\sse}{\subseteq}

\newcommand\ie{i.e., }
\newcommand\Wlog{W.l.o.g.~}

\newcommand\eg{e.g.\xspace}

\newcommand{\pref}{\leq_{\text{pref}}}
\newcommand{\suff}{\leq_{\text{suff}}}
\newcommand{\spref}{<_{\text{pref}}}
\newcommand{\ssuff}{<_{\text{suff}}}

 \renewcommand{\theenumi}{\textup{(\roman{enumi})}}
 
 \renewcommand{\theenumiii}{\textup{\arabic{enumiii}.}}

\newcommand{\y}{y}

\tikzset{
	ncbar angle/.initial=90,
	ncbar/.style={
		to path=(\tikztostart)
		-- ($(\tikztostart)!#1!\pgfkeysvalueof{/tikz/ncbar angle}:(\tikztotarget)$)
		-- ($(\tikztotarget)!($(\tikztostart)!#1!\pgfkeysvalueof{/tikz/ncbar angle}:(\tikztotarget)$)!\pgfkeysvalueof{/tikz/ncbar angle}:(\tikztostart)$)
		-- (\tikztotarget)
	},
	ncbar/.default=0.5cm,
}

\newcommand{\PowWP}{\textsc{PowerWP}}
\newcommand{\WP}{\textsc{WP}}
\newcommand{\MEM}{\textsc{Membership}}

\title{The power word problem} 

\author{Markus Lohrey}{Universit{\"a}t Siegen, Germany }{}{}{Funded by DFG project LO 748/12-1.}

\author{Armin Wei\ss
}{Universit{\"a}t Stuttgart, Germany}{}{}{Funded by DFG project DI 435/7-1.} 

\authorrunning{M. Lohrey and A. Wei\ss}

\Copyright{Markus Lohrey and Armin Wei\ss}



\category{}

\relatedversion{}

\supplement{}

\acknowledgements{We thank Laurent Bartholdi for pointing out the result \cite[Theorem~6.6]{BartholdiGS03} on the bound of the order of elements in the Grigorchuk group, which allowed us to establish \prettyref{thm:grigorchuck}.}

\nolinenumbers 

\EventEditors{John Q. Open and Joan R. Access}
\EventNoEds{2}
\EventLongTitle{42nd Conference on Very Important Topics (CVIT 2016)}
\EventShortTitle{CVIT 2016}
\EventAcronym{CVIT}
\EventYear{2016}
\EventDate{December 24--27, 2016}
\EventLocation{Little Whinging, United Kingdom}
\EventLogo{}
\SeriesVolume{42}
\ArticleNo{23}

\begin{document}
	
	\maketitle
	
	\begin{abstract}
	In this work we introduce a new succinct variant of the word problem in a finitely generated group $G$, which we call the power word problem: the input word may contain powers $p^x$, where $p$ is a finite word
	over generators of $G$ and $x$ is a binary encoded integer. The power
	word problem is a restriction
	of the compressed word problem, where the input word is represented by 
	a straight-line program (i.e., an algebraic circuit over $G$). The main result of the paper
	states that the power word problem for a finitely generated free group $F$ is 
	\Ac{0}-Turing-reducible to the word problem for $F$.
	Moreover, the following hardness result is
	shown: For a wreath product $G \wr \mathbb{Z}$, where $G$ is either free of rank at least two
	or finite non-solvable, the power word problem is complete for \coNP. This contrasts with the situation where $G$ is abelian: then the power word problem is shown to be in \Tc0.
	\end{abstract}


  \section{Introduction}
  
  Algorithmic problems in group theory have a long tradition, going back to the work of Dehn from 1911 \cite{dehn11}.
  One of the fundamental group theoretic decision problems introduced by Dehn is the {\em word problem}
  for a finitely generated group $G$ (with a fixed finite generating set $\Sigma$): does a given word $w \in \Sigma^*$
  evaluate to the group identity? Novikov \cite{nov55} and Boone \cite{boone59} independently proved in the 1950's the existence 
  of finitely presented groups with undecidable word problem. On the positive side, in many important classes
  of groups the word problem is decidable, and in many cases also the computational complexity is quite low. 
  Famous examples are finitely generated linear groups, where the word problem can be solved in logarithmic
  space \cite{lz77} and hyperbolic groups where the word problem can be solved in linear time \cite{Ho} as well as in \LOGCFL
   \cite{Lo05ijfcs}. 
   
   In recent years, also compressed versions of group theoretical decision problems,
   where input words are represented in a succinct form, have attracted attention. One such succinct representation
   are so-called straight-line programs, which are context-free grammars that produce exactly one word. The size 
   of such a grammar can be much smaller than the word it produces. 
   For instance, the word $a^n$ can be produced by a straight-line program of size $\mathcal{O}(\log n)$.
   For the {\em compressed word problem} for the group $G$ the input consists of a straight-line program that produces a word $w$
   over the generators of $G$ and it is asked whether $w$ evaluates to the identity element of $G$. 
   This problem is a reformulation of the circuit evaluation problem for $G$. The compressed word problem 
   naturally appears when one tries to solve the word problem in automorphism groups or semidirect 
   products \cite[Section~4.2]{LohreySpringerbook2014}.
   For the following classes of groups, the compressed word problem is known 
   to be solvable in polynomial time: finite groups (where the compressed word problem is either \P-complete or in \Nc{2}
   \cite{BeMcPeTh97}), finitely generated nilpotent groups \cite{KoenigL17} (where the complexity is even in \Nc{2}),
   hyperbolic groups \cite{HoLoSchl19} (in particular, free groups), and  virtually special groups 
   (i.e, finite extensions of subgroups of right-angled Artin groups) \cite{LohreySpringerbook2014}.
   The latter class covers for instance Coxeter groups, one-relator groups with torsion, fully residually free groups and 
   fundamental groups of hyperbolic 3-manifolds. For finitely generated linear groups there is still a randomized polynomial
   time algorithm for the compressed word problem \cite{LohreyS07,LohreySpringerbook2014}.
   Simple examples of groups where the compressed word problem is intractable are wreath products $G \wr \mathbb{Z}$ with $G$ a non-abelian group: for every such group the 
   compressed word problem is \coNP-hard \cite{LohreySpringerbook2014} (this includes for instance Thompson's group $F$); on the other hand, if, in addition, $G$ is finite, then the (ordinary) word problem for $G \wr \mathbb{Z}$ is in $\Nc{1}$ \cite{Waack90}.
   
   In this paper, we study a natural variant of the compressed word problem, called the {\em power word problem}. 
   An input for the power word problem for the group $G$ is a tuple $(p_1, x_1, p_2, x_2, \ldots, p_n, x_n)$ where every $p_i$ 
   is a word over the group generators and every $x_i$ is a binary encoded integer (such a tuple is called a {\em power word}); 
   the question is whether 
   $p_1^{x_1} p_2^{x_2}\cdots p_n^{x_n}$ evaluates to the group identity of $G$. From a power word $(p_1, x_1, p_2, x_2, \ldots, p_n, x_n)$ 
   one can easily (\eg\ by an $\Ac{0}$-reduction) compute a straight-line program for the word 
   $p_1^{x_1} p_2^{x_2}\cdots p_n^{x_n}$. In this sense, the power word problem is at most as difficult as the compressed
   word problem. On the other hand, both power words and straight-line programs achieve exponential compression in the best case; so the additional difficulty of the the compressed word problem does not come from a higher compression rate but rather because straight-line programs can generate more ``complex'' words. 
   
   Our main results for the power word problem are the following; in each case we compare our results with
   the corresponding results for the compressed word problem:
   \begin{itemize}
   \item The power word problem for every finitely generated nilpotent group is in  \DLOGTIME-uniform $\Tc{0}$
   and hence has the same complexity as the (ordinary) word problem (or the problem of multiplying binary encoded integers). 
   The proof is a straightforward adaption 
   of a proof from \cite{MyasnikovW17}. There, the special case, where all $p_i$ in the input power word 
   $(p_1, x_1, p_2, x_2, \ldots, p_n, x_n)$ are single generators, was shown to be in \DLOGTIME-uniform $\Tc{0}$. 
   The compressed word problem for every finitely generated nilpotent group belongs to the 
   class $\mathsf{DET} \subseteq \Nc{2}$ and is hard for the counting class $\mathsf{C}_{=}\mathsf{L}$ in case
    of a torsion-free nilpotent group \cite{KoenigL17}.
   \item  The power word problem for a finitely generated group $G$ is \Nc{1}-many-one-reducible to the power word problem for any finite index subgroup of $G$. An analogous result holds for the compressed word problem as well \cite{KoenigL17}.
   \item The power word problem for a finitely generated free group is \Ac{0}-Turing-reducible to the word problem for 
   $F_2$ (the free group  of rank two) and therefore belongs to logspace. In contrast, it was shown in \cite{Lohrey06siam} that 
   the compressed word problem for a finitely generated free group
   of rank at least two is \P-complete.
   \item The power word problem for a wreath product $G \wr \mathbb{Z}$ with $G$ finitely generated abelian belongs to 
   \DLOGTIME-uniform $\Tc{0}$.
   For the compressed word problem for $G \wr \mathbb{Z}$ with $G$ finitely generated abelian only 
   the existence of a randomized polynomial time algorithm is known \cite{KonigL18}.
   \item The power word problem for the wreath products $F_2 \wr \mathbb{Z}$ and every wreath product 
   $G \wr \mathbb{Z}$, where $G$ is finite and non-solvable, is \coNP-complete. For these groups this sharpens the 
   corresponding \coNP-hardness result for the compressed word problem \cite{LohreySpringerbook2014}.
  \end{itemize}
\ifFullVersion\else
Due to space constraints we present only short outlines of the proofs for our main theorems; the proofs of all lemmas can be found in the appendix or in our full version on arXiv \cite{LohreyW19arxiv}.\todo{Hier oder wo anders? Paper auf Axiv Referenzieren?}
\fi

  \subparagraph*{Related work.}
  Implicitly, (variants of) the power word problem have been studied before.
  In the commutative setting, Ge \cite{Ge93} has shown that one can verify in polynomial time an identity 
  $\alpha_1^{x_1} \alpha_2^{x_2}\cdots \alpha_n^{x_n} = 1$, where the $\alpha_i$ are elements of an algebraic
  number field and the $x_i$ are binary encoded integers.
  
  Another problem related to the power word problem is the knapsack problem \cite{GanardiKLZ18,LohreyZ18,MyNiUs14} for a finitely generated group $G$
  (with generating set $\Sigma$): for a given sequence 
  of words $w, w_1, \ldots, w_n \in \Sigma^*$, the question is whether there exist $x_1, \ldots, x_n \in \mathbb{N}$ 
  such that $w = w_1^{x_1} \cdots w_n^{x_n}$ holds in $G$. For many groups $G$ one can show that 
  if such $x_1, \ldots, x_n \in \mathbb{N}$ exists, then there exist such numbers of size $2^{\text{poly}(N)}$, where
  $N = |w|+|w_1|+\cdots+|w_n|$ is the input length. This holds for instance for right-angled Artin groups (also known
  as graph groups). In this case, one nondeterministically guesses the binary encodings of numbers $x_1, \ldots, x_n$
  and then verifies, using an algorithm for the power word problem, whether $w_1^{x_1} \cdots w_n^{x_n} w^{-1} = 1$
  holds. In this way, it was shown in \cite{LohreyZ18} that for every right-angled Artin group the knapsack problem belongs to \NP
  (using the fact that the compressed word problem and hence the power word problem for a right-angled Artin group
  belongs to \P).
 
 In \cite{GurevichS07}, Gurevich and Schupp present a polynomial time 
 algorithm for a compressed form of the subgroup membership problem for a free group $F$,
 where group elements are represented in the form $a_1^{x_1} a_2^{x_2}\cdots a_n^{x_n}$ with binary encoded integers $x_i$
 The $a_i$ must be standard generators of the free group $F$. This is the same input representation as in 
 \cite{MyasnikovW17} and is more restrictive then our setting, where
 we allow powers of the form $w^x$ for $w$ an arbitrary word over the group generators (on the other hand, 
 Gurevich and Schupp consider the subgroup membership problem, which is more general than the word problem).

\section{Preliminaries}\label{sec:prelims} 
\ifFullVersion We denote intervals of integers with $[a,b] = \set{x \in \Z}{a \leq x \leq b}$.\fi
 
 \ifFullVersion \vspace{-1mm}\fi
\subparagraph*{Words.} An \emph{alphabet} is a (finite or infinite) set $\Sig$; an element $a \in \Sig$ is called a \emph{letter}. 
The free monoid over $\Sig$ is denoted by $\Sig^*$, its elements
are called {\em words}. The multiplication of the monoid is concatenation of words. The identity element is the empty word $1$. 
The length of a word $w$ is denoted by $\abs w$. 
If $w,p,x,q$ are words such that $w = pxq$, then we call $x$ a \emph{factor} of $w$, $p$ a \emph{prefix} of $w$, and $q$ a \emph{suffix} of $w$. 
We write $v \pref w$ (resp.\ $v \spref w$) if $v$ is a (strict) prefix of $w$ and $v \suff w$ (resp.\ $v \ssuff w$) if $v$ is a (strict) suffix of $w$.

\newcommand\SGr{\Sig^*}

\subparagraph*{String rewriting systems.}
Let $\Sig$ be an alphabet and $S \sse \SGr \times \SGr$ be a set of pairs, called a \emph{string rewriting system}.
 We write $\ell \to r$ if $(\ell, r) \in S$. This corresponding \emph{rewriting relation} $\smash{\RAS{}S} \vphantom{\DAS{}{}}$ over $\SGr$ is 
 defined by: $u \RAS{}S v$ if and only if there exist $\ell\to r\in S$ and words $s,t \in \Sig^*$ such that 
 $u = s\ell t$ and $v = sr t$. We also say that $u$ can be rewritten to $v$ in one step.
 We write $u \RAS{k}S v$ if $u$ can be rewritten to $v$ in exactly $k$ steps, \ie if there are $u_0, \dots, u_k$ with $u = u_0$, $v=u_k$ and $u_i\RAS{}Su_{i+1} $ for $0 \leq i \leq k-1$.
We denote the transitive closure of $\RAS{}S$ by $\RAS{+}S \;=\; \bigcup_{k\geq 1}\RAS{k}S$ and the reflexive and transitive closure by $\RAS{*}S\;=\; \bigcup_{k\geq 0}\RAS{k}S$. Moreover $\DAS*S$ is the reflexive, transitive, and symmetric closure of  $\RAS{}S$; it is the smallest congruence containing $S$. The set of \emph{irreducible word} with respect to $S$
is 
$\IRR(S) = \{w \in \Sigma^* \mid \text{there is no } v \text{ with } w \RAS{}S v\}$.

\subparagraph*{Free groups.}

Let $X$ be a set and $X^{-1} = \set{a^{-1}}{a \in X}$ be a disjoint copy of $X$.
We extend the mapping $a \mapsto a^{-1}$ 
 to an involution without fixed points on $\Sig = X \cup X^{-1}$ by $(a^{-1})^{-1} = a$
 and finally to an involution without fixed points on $\Sig^*$ by $(a_1 a_2 \cdots a_n)^{-1} = 
 a_n^{-1} \cdots a_2^{-1} a_1^{-1}$. For an integer $z < 0$ and $w \in \Sig^*$ we 
 write $w^z$ for $(w^{-1})^{-z}$.
 The string rewriting system \ifFullVersion
 \[S= \set{a a^{-1} \to 1}{a \in \Sig}\]
 \else
  $S= \set{a a^{-1} \to 1}{a \in \Sig}$
 \fi
is strongly confluent and terminating meaning that for every word $w \in \Sig^*$ there exists
a unique word $\red{w} \in \IRR(S)$ with $w \RAS*S \red{w}$ (for precise definitions see \eg\ \cite{bo93springer,jan88eatcs}).
Words from $\IRR(S)$ are called \emph{freely reduced}.
The system $S$ defines the free group $F_{X} = \Sig^*/ S$ with basis $X$.
More concretely, elements of $F_X$ can be identified with freely reduced normal forms, and 
the group product of $u, v \in \IRR(S)$ is defined by $\red{uv}$. With this definition 
$\mathrm{red}: \Sig^* \to F_X$ becomes a monoid homomorphism that commutes with the 
involution ${\cdot}^{-1}$: $\red{w}^{-1} = \red{w^{-1}}$ for all words $w \in \Sig^*$.
If $|X|=2$ then we write $F_2$ for $F_X$. It is known that for every countable set $X$,
$F_2$ contains an isomorphic copy of $F_X$.

\subparagraph*{Finitely generated groups and the power word problem.}

A group $G$ is called {\em finitely generated} if there exist a finite a finite set $X$ and 
a surjective group homomorphism $h : F_X \to G$. In this situation, the set 
$\Sig = X \cup X^{-1}$ is called a finite (symmetric) generating set for $G$.
\ifFullVersion In many cases we can think of $\Sigma$ as a subset of $G$, but, in general, 
we can also have more than one letters for the same group element. 
The group identity of $G$ is denoted with $1$ as well
  (this fits to our notation $1$
for the empty word which is the identity of $F_X$).\fi
For words $u,v \in \Sigma^*$  we usually say that $u = v$ in $G$ or $u =_G v$ in case $h(\red{u}) = h(\red{v})$. 
The {\em word problem} for the finitely generated group $G$, $\WP(G)$ for short, is defined as follows:
\begin{itemize}
\item input: a word $w \in \Sigma^*$.
\item question: does $w=_G 1$ hold? 
\end{itemize}
A \emph{power word} (over $\Sigma$) is a tuple $(p_1,x_1,p_2,x_2,\ldots,p_n,x_n)$ where 
$p_1, \dots, p_n \in \Sigma^*$ are words over the group generators (called the periods of the power word)
and $x_1, \dots, x_n\in \Z$ are integers that are given in binary notation. Such a power word represents the 
word $p_1^{x_1} p_2^{x_2}\cdots p_n^{x_n}$. Quite often, we will identify the power word $(p_1,x_1,p_2,x_2,\ldots,p_n,x_n)$
with the word $p_1^{x_1} p_2^{x_2}\cdots p_n^{x_n}$. Moreover, if $x_i=1$, then we usually omit the exponent $1$ in a power word. 
The \emph{power word problem} for the finitely generated group $G$, $\PowWP(G)$ for short, is defined as follows:
\begin{itemize}
\item input: a power word $(p_1,x_1,p_2,x_2,\ldots,p_n,x_n)$.
\item question: does $p_1^{x_1} p_2^{x_2}\cdots p_n^{x_n}=_G 1$ hold?
\end{itemize}
Due to the binary encoded exponents, a power word can be seen as a succinct description 
of an ordinary word. Hence, a priori, the power word problem for a group $G$ could be computationally
more difficult than the word problem. We will see examples of groups $G$, where 
$\PowWP(G)$ is indeed more difficult than $\WP(G)$ (under standard assumptions from complexity theory),
as well as examples of groups $G$, where $\PowWP(G)$ and $\WP(G)$ are equally difficult.

\subparagraph*{Wreath products.}

Let $G$ and $H$ be groups. Consider the direct sum $K = \bigoplus_{h
  \in H} G_h$, where $G_h$ is a copy of $G$. We view $K$ as the set $G^{(H)}$ of
all mappings $f\colon H\to G$ such that $\supp(f) := \{h\in H \mid f(h)\ne
1\}$ is finite, together with pointwise multiplication as the group
operation.  The set $\supp(f)\subseteq H$ is called the
\emph{support} of $f$. The group $H$ has a natural left action on
$G^{(H)}$ given by $h f(a) = f(h^{-1}a)$, where $f \in G^{(H)}$ and
$h, a \in H$.  The corresponding semidirect product $G^{(H)} \rtimes
H$ is the (restricted) \emph{wreath product} $G \wr H$.  In other words:
\begin{itemize}
\item
Elements of $G \wr H$ are pairs $(f,h)$, where $h \in H$ and
$f \in G^{(H)}$.
\item
The multiplication in $G \wr H$ is defined as follows:
Let $(f_1,h_1), (f_2,h_2) \in G \wr H$. Then
$(f_1,h_1)(f_2,h_2) = (f, h_1h_2)$, where
$f(a) = f_1(a)f_2(h_1^{-1}a)$.
\end{itemize}
\ifFullVersion
The following intuition might be helpful:
An element $(f,h) \in G\wr H$ can be thought of
as a finite multiset of elements of $G \setminus\{1_G\}$ that are sitting at certain
elements of $H$ (the mapping $f$) together with the distinguished
element $h \in H$, which can be thought of as a cursor
moving in $H$.
If we want to compute the product $(f_1,h_1) (f_2,h_2)$, we do this
as follows: First, we shift the finite collection of $G$-elements that
corresponds to the mapping $f_2$ by $h_1$: If the element $g \in G\setminus\{1_G\}$ is
sitting at $a \in H$ (i.e., $f_2(a)=g$), then we remove $g$ from $a$ and
put it to the new location $h_1a \in H$. This new collection
corresponds to the mapping $f'_2 \colon  a \mapsto f_2(h_1^{-1}a)$.
After this shift, we multiply the two collections of $G$-elements
pointwise: If in $a \in H$ the elements $g_1$ and $g_2$ are sitting
(i.e., $f_1(a)=g_1$ and $f'_2(a)=g_2$), then we put the product
$g_1g_2$ into the location $a$. Finally, the new distinguished
$H$-element (the new cursor position) becomes $h_1 h_2$.
\fi

\subparagraph*{Complexity.}

\ifFullVersion
We assume that the reader is familiar with the complexity classes {\sf P}, \NP, and \coNP; see e.g.\ \cite{AroBar09} for details. Let $\cC$ be any complexity class and $K\sse \Delta^*$, $L \sse \Sigma^*$ languages. Then $L$ is $\cC$-many-one reducible to $K$ ($L\leq_{\mathrm{m}}^\cC  K$) if there exists a $\cC$-computable function $f: \Sigma^* \to \Delta^*$ with $x\in L $ if and only if $f(x) \in K$.

We use circuit complexity for classes below deterministic logspace (\L for short).
Instead of defining these classes directly, we introduce the slightly more general notion of \Ac0-Turing reducibility.
A language $L \subseteq \{0,1\}^*$  is \Ac0-Turing-reducible to $K \sse \{0,1\}^*$ if there is a family of constant-depth, polynomial-size Boolean circuits with oracle gates for $K$ deciding $L$. More precisely, we can define the class of language $\Ac0(K)$ which are \Ac0-Turing-reducible to $K \sse \{0,1\}^*$: a language $L \subseteq \{0,1\}^*$ belongs to $\Ac0(K)$ if there exists a family $(C_n)_{n \geq 0}$ of Boolean circuits with the following 
properties:
\begin{itemize}
	\item $C_n$ has $n$ distinguished input gates $x_1, \ldots, x_n$ and a distinguished output gate $o$.
	\item $C_n$ accepts exactly the words from $L \cap \{0,1\}^n$, i.e., if the input gate $x_i$ receives the input $a_i \in \{0,1\}$, then
	the output gate $o$ evaluates to $1$ if and only if $a_1 a_2 \cdots a_n \in L$.
	\item Every circuit $C_n$ is built up from input gates, not-gates, and-gates, or-gates, and oracle gates for $K$ (which output $1$ if and only if their input is in $K$).
	\item All gates have unbounded fan-in, \ie there is no bound on the number of input wires for a gate.
	\item There is a polynomial $p(n)$ such that $C_n$ has at most $p(n)$ many gates and wires.
	\item There is a constant $d$ such that every $C_n$ has depth at most $d$, where the depth is the length of a longest path
	from an input gate $x_i$ to the output gate $o$.
\end{itemize}
This is in fact the definition of non-uniform $\Ac0(K)$. Here ``non-uniform'' means that the mapping $n \mapsto C_n$ is 
not restricted in any way. In particular, it can be non-computable. For algorithmic purposes one usually adds some uniformity
requirement to the above definition. The most ``uniform'' version of $\Ac0(K)$ is \DLOGTIME-uniform $\Ac0(K)$. For this,
one encodes the gates of each circuit $C_n$ by bit strings of length $\mathcal{O}(\log n)$. Then the circuit family $(C_n)_{n \geq 0}$
is called \emph{\DLOGTIME-uniform}  if (i) there exists a deterministic Turing machine that computes for a given gate $u \in \{0,1\}^*$ 
of $C_n$ ($|u| \in \mathcal{O}(\log n)$) in time $\mathcal{O}(\log n)$ the type (of gate $u$, where the types are $x_1, \ldots, x_n$, not, and, or, oracle gates)
and (ii) there exists a deterministic Turing machine that decides for two given gate $u,v \in \{0,1\}^*$ 
of $C_n$ ($|u|, |v| \in \mathcal{O}(\log n)$) in time $\mathcal{O}(\log n)$ whether there is a wire from gate $u$ to gate $v$.
In the following, we write $\AC(K)$ for \DLOGTIME-uniform $\Ac0(K)$.

If the language $L$ in the above definition of $\AC(K)$ is defined over a non-binary alphabet $\Sigma$, then one first has to fix a binary 
encoding of words over $\Sigma$.

The class \Nc1 is defined as the class of languages accepted by boolean circuits of bounded fan-in and logarithmic depth. As a consequence of Barrington's theorem \cite{Barrington86}, we have $\Nc1 = \AC(\WP(A_5))$ where $A_5$ is the alternating group over 5 elements \cite[Corollary 4.54]{Vollmer99}. Moreover, the word problem for any finite group $G$ is in $\Nc1$; if $G$ is non-solvable, its word problem is $\Nc1$-complete~-- even under $\AC$-many-one reductions. Robinson proved that the word problem for the free group $F_2$ is \Nc1-hard  \cite{Robinson93phd}, i.e., $\Nc1 \subseteq \AC(\WP(F_2))$.

The class $\TC$ is defined as $\AC(\text{\textsc{Majority}})$ where \textsc{Majority} is the problem to determine whether the input contains more $1$s than $0$s. When talking about hardness for \TC or \Nc1 we use \AC-Turing reductions unless stated otherwise.
Important problems that are complete for \TC are:
\begin{itemize}
	\item The languages $\{ w \in \{0,1\}^* \mid |w|_0 \leq |w|_1 \}$ and $\{ w \in \{0,1\}^* \mid |w|_0 = |w|_1 \}$, where $|w|_a$ denotes
	the number of occurrences of $a$ in $w$, see e.g. \cite{Vollmer99}.
	\item The computation (of a certain bit) of the binary representation of the product of two (or any number of) 
	binary encoded integers \cite{HeAlBa02}.
	\item The computation (of a certain bit) of the binary representation of the integer quotient of two binary encoded integers \cite{HeAlBa02}.
	\item The word problem for every infinite solvable linear group \cite{KoenigL17}.
	\item The conjugacy problem for the Baumslag-Solitar group  $\mathsf{BS}(1,2)$  \cite{DiekertMW14}.
\end{itemize}

\else

We assume that the reader is familiar with the complexity classes {\sf P}, \NP, and \coNP and many-one reductions; see e.g.\ \cite{AroBar09} for details. 
We use circuit complexity for classes below deterministic logspace (\L for short).

A language $L \subseteq \{0,1\}^*$  is \Ac0-Turing-reducible to $K \sse \{0,1\}^*$ if there is a family of constant-depth, polynomial-size Boolean circuits with oracle gates for $K$ deciding $L$. More precisely, $L \subseteq \{0,1\}^*$ belongs to $\Ac0(K)$ if there exists a family $(C_n)_{n \geq 0}$ of circuits which, apart from the input gates $x_1, \ldots, x_n$ are built up from \emph{not}, \emph{and}, \emph{or}, and \emph{oracle gates} for $K$ (which output $1$ if and only if their input is in $K$).
All gates may have unbounded fan-in, but there is a polynomial bound on the number of gates and wires and a constant bound on the depth (length of a longest path from an input gate $x_i$ to the output gate $o$).
 Finally, $C_n$ accepts exactly the words from $L \cap \{0,1\}^n$, i.e., if each input gate $x_i$ receives the input $a_i \in \{0,1\}$, then a
 distinguished output gate evaluates to $1$ if and only if $a_1 a_2 \cdots a_n \in L$.

In the following, we only consider \DLOGTIME-uniform $\Ac0(K)$ for which we write $\AC(K)$.
\DLOGTIME-uniform means that there is a deterministic Turing machine which decides in time $\Oh(\log n)$ on input of two gate numbers (given in binary) and the string $1^n$ whether there is a wire between the two gates in the $n$-input circuit and also computes the type of a given gates.
For more details on these definitions we refer to \cite{Vollmer99}. 
%
%
If the languages $K$ and $L$ in the above definition of $\AC(K)$ are defined over a non-binary alphabet $\Sigma$, then one first has to fix a binary 
encoding of words over $\Sigma$. 

The class $\TC$ is defined as $\AC(\text{\textsc{Majority}})$ where \textsc{Majority} is the problem to determine whether the input contains more $1$s than $0$s. 
The class \Nc1 is the class of languages accepted by Boolean circuits of bounded fan-in and logarithmic depth. 
When talking about hardness for \TC or \Nc1 we use \AC-Turing reductions unless stated otherwise.
As a consequence of Barrington's theorem \cite{Barrington86}, we have $\Nc1 = \AC(\WP(A_5))$ where $A_5$ is the alternating group over 5 elements \cite[Corollary 4.54]{Vollmer99}.
Moreover, the word problem for any finite group $G$ is in $\Nc1$.
Robinson proved that the word problem for the free group $F_2$ is \Nc1-hard  \cite{Robinson93phd}, i.e., $\Nc1 \subseteq \AC(\WP(F_2))$.

\fi

\ifFullVersion
Let $A, B_1, \ldots, B_k \subseteq \{0,1\}^*$ be languages. We say that $A$ is {\em conjunctive truth-table \TC-reducible} to $B_1, \ldots, B_k$
if there exists a \TC-computable function $f$ that computes from a given input word $w \in \{0,1\}^*$ a finite list
$w_1, i_1, w_2, i_2, \ldots, w_d, i_d$ with $w_1, \ldots, w_d \in \{0,1\}^*$ and $i_1, \ldots, i_d \in \{1, \ldots, k\}$ such that
$w \in A$ if and only if $\bigwedge_{1 \leq j \leq d} w_j \in B_{i_j}$.
We need the following obvious fact (for which conjunctive truth-table polynomial time reducibility would suffice):

\begin{lemma} \label{lem:conjunctive-truth-table}
If $A$ is conjunctive truth-table \TC-reducible to $B_1, \ldots, B_k$ and $B_1, \ldots, B_k$ belong to \coNP then also
$A$ belongs to \coNP.
\end{lemma}
\fi

\section{Results}\label{sec:results}

In this section we state our (and proof the easy) results on the power word problem. 
\ifFullVersion 
	The proofs of \prettyref{thm:free_power_wp}, \ref{thm:tc0-wreath} and \ref{thm:wreath-coNP} can be found in Sections \ref{sec:proof_free}, \ref{sec:proof-tc0-wreath}, and \ref{sec:proof-wreath-coNP} respectively.
\else
	Outlines of the proofs of \prettyref{thm:free_power_wp}, \ref{thm:tc0-wreath} and \ref{thm:wreath-coNP} can be found in Sections \ref{sec:proof_free} and \ref{sec:proof-wreath}, respectively.
\fi

\begin{theorem}\label{thm:nilpotent_power_wp}
If $G$ is a finitely generated nilpotent group, then $\PowWP(G)$ is in \TC.
\end{theorem}
\begin{proof}
	In \cite{MyasnikovW17}, the so-called word problem with binary exponents was shown to be in \TC. We can apply the same techniques as in  \cite{MyasnikovW17}: we compute Mal'cev normal forms of all $p_i$ \cite[Theorem 5]{MyasnikovW17}, then use the power polynomials from \cite[Lemma 2]{MyasnikovW17} to compute Mal'cev normal forms with binary exponents of all $p_i^{x_i}$. Finally, we compute the Mal'cev normal form of $p_1^{x_1} \cdots p_n^{x_n}$ again using \cite{MyasnikovW17}.
\end{proof}

\begin{theorem}\label{thm:free_power_wp}
	The power word problem for a finitely generated free group is \Ac{0}-Turing-reducible to the word problem for the free group  $F_2$.
\end{theorem}
Notice that if the free group has rank one, then the power word problem is in \TC because iterated addition is in \TC.

\begin{remark}
	If the input is of the form $(p_1,x_1,p_2,x_2,\ldots,p_n,x_n)$ where all $p_i$ are freely reduced, then the reduction in \prettyref{thm:free_power_wp} is a \TC-many-one reduction.
\end{remark}

\begin{remark}
 One can consider variants of the power word problem, where the exponents are not given in binary representation but in even more
 compact forms. \emph{Power circuits} as defined in \cite{MyasnikovUW12} are such a representation that allow 
non-elementary compression for some integers. The proof of 
\prettyref{thm:free_power_wp} involves iterated addition and comparison of exponents.
For power circuits iterated addition is in \AC (just putting 
the power circuits next to each other), but comparison (even for equality) is \P-complete \cite{Weiss15diss}. 
Hence, the variant of the power word problem, where exponents are encoded with power circuits is \P-complete.
\end{remark}
\begin{remark}
	The proof of \prettyref{thm:free_power_wp} can be easily generalized to free products. However, in order to have a simpler presentation we only state and prove the result for free groups and postpone the free product case to a future full version.
\end{remark}
It is easy to see that the power word problem for every finite group belongs to \Nc{1}. The following result generalizes this fact:

\begin{theorem}\label{thm:finite_index}
Let $G$ be finitely generated and let $H\leq G$ have finite index. Then $\PowWP(G)$ is \Nc{1}-many-one-reducible to $\PowWP(H)$. 
\end{theorem}

\ifFullVersion

\begin{proof}
If $H \leq G$ is of finite index, then there is a normal subgroup $N \leq G$ of finite index and $N\leq  H$ (\eg $N= \bigcap_{g\in G} gHg^{-1}$). As $N \leq H$, $\PowWP(N)$ is reducible via a homomorphism (\ie in particular in \TC) to $\PowWP(H)$. 
Thus, we can assume that from the beginning $H$ is normal and that $Q=G/H$ is a finite quotient group. 
Notice that $H$ is finitely generated as $G$ is so.
Let $R\sse G$ denote a set of representatives of $Q$ with $1 \in R$. If we choose a finite generating set $\Sigma$ for $H$, then
$\Sigma \cup R \setminus \{1\}$ becomes a generating set for $G$.

As a first step for every exponent $x_i$ in the input power word we compute numbers $y_i,z_i$ with
$x_i = y_i\abs{Q} + z_i$ and $0 \leq z_i < \abs{Q}$ (\ie we compute the division with remainder by $\abs{Q}$).
This is possible in \Nc{1} \cite{HeAlBa02}. 
Note that $p_i^{\abs{Q}}$ is trivial in the quotient $Q = G/H$ and therefore represents an element of $H$.
Using the conjugate collection process from \cite[Theorem~5.2]{Robinson93phd} we can compute in $\Nc{1}$ a
word $h_i \in \Sigma^*$ such that $p_i^{\abs{Q}} =_G h_i$. 
Then we replace in the input word every $p_i^{x_i}$ by $h_i^{y_i} p_i^{z_i}$ where we write $p_i^{z_i}$ as a word without exponents.
We have obtained a word where all factors with exponents represent elements of $H$. Finally, we proceed like Robinson \cite{Robinson93phd} for the ordinary word problem 
treating words with exponents as single letters (this is possible because they are in $H$).

To give some more details, let us denote the result of the previous step as
$g_0 h_1^{y_1} g_1 \cdots h_n^{y_n} g_n$ with $g_i \in (\Sigma \cup R \setminus \{1\})^*$. By \cite[Theorem~5.2]{Robinson93phd} we can rewrite $g_i$ as $g_i = \tilde h_i r_i$ with $r_i \in R$ and $\tilde h_i \in \Sigma^*$. 
Once again, we follow \cite{Robinson93phd} and write $\tilde h_0 r_0 h_1^{y_1} \tilde h_1 r_1 \cdots h_n^{y_n} \tilde h_n r_n$ as
\[
\tilde h_0 w_0 (a_1 h_1^{y_1} \tilde h_1a_1^{-1}) w_1  (a_2 h_2^{y_2} \tilde h_2a_2^{-1})w_2  \cdots (a_n h_n^{y_n} \tilde h_na_n^{-1})w_na_{n+1}
\]
where $a_i$ is the representative of $r_0 \cdots r_{i-1}$ in $R$ and $w_i = a_i r_i a_{i+1}^{-1}$. The conjugation $(a_i h_i^{y_i} \tilde h_ia_i^{-1})$ is an application of one of a fixed finite set of homomorphisms and, thus, can be computed in $\TC$. Notice that $w_i \in H$ for all $i$ and, as it comes from a fixed finite set (namely $R \cdot R \cdot R^{-1}$), each $w_i$ can be rewritten to $w_i' \in \Sigma^*$. Now it remains to verify whether $a_{n+1}= 1$ (in $\NC^1$). If this is not the case, we output any non-identity word in $H$, otherwise we output
$\tilde h_0 w_0' (a_1 h_1^{y_1} \tilde h_1a_1^{-1}) w_1'  (a_2 h_2^{y_2} \tilde h_2a_2^{-1})w_2'  \cdots (a_n h_n^{y_n} \tilde h_na_n^{-1})w_n'$.
\end{proof}
\else
\begin{proof}[Proof sketch]
\Wlog we can assume that $H$ is a finitely generated normal subgroup and $R$ is a finite set of representatives of $Q := G/H$ with $1\in R$.
As a first step we replace in the input power word every $p_i^{x_i}$ by $h_i^{y_i} p_i^{z_i}$ where $x_i = y_i\abs{Q} + z_i$, $0 \leq z_i < \abs{Q}$ and $h_i$ is a word over the generators of $H$ with $p_i^{\abs{Q}} \!=_G\! h_i$. Moreover, we write $p_i^{z_i}$ as a word without exponents.
Using the conjugate collection process from \cite[Theorem~5.2]{Robinson93phd}, the result can be rewritten in the form $h r$ where $h$ is a power word in the subgroup $H$ and $r\in R$.
\end{proof}
\fi

As an immediate consequence of \prettyref{thm:free_power_wp}, \prettyref{thm:finite_index} and the \Nc{1}-hardness 
of the word problem for $F_2$ \cite[Theorem~6.3]{Robinson93phd} we obtain:
\begin{corollary}
	The power word problem for every finitely generated virtually free group is \Ac{0}-Turing-reducible to the word problem for the free group $F_2$.
\end{corollary}

\begin{theorem} \label{thm:tc0-wreath}
For every finitely generated abelian group $G$,  $\PowWP(G \wr \Z)$ \ifFullVersion belongs to \else is in \fi $\TC$.
\end{theorem}

\begin{theorem} \label{thm:wreath-coNP}
Let $G$ be either a finite non-solvable group or a finitely generated free group of rank at least two.
Then $\PowWP(G \wr \Z)$ is  \coNP-complete.
\end{theorem}

\begin{theorem}\label{thm:grigorchuck}
	The power word problem for the Grigorchuk group (as defined in \cite{Grig80} and also known as \emph{first Grigorchuk group}) is \AC-many-one-reducible to its word problem. 
\end{theorem}
\prettyref{thm:grigorchuck} applies only if the generating set contains a neutral letter. Otherwise, the reduction is in \TC. It is well-know that the word problem for the Grigorchuk group is in \L (see \eg\ \cite{Nekrashevych05,MiasnikovV17}). Thus, also the power word problem is in \L. 

\ifFullVersion
\begin{proof}[Proof of \prettyref{thm:grigorchuck}.]
	Let $G$ denote the Grigorchuk group.
	By \cite[Theorem~6.6]{BartholdiGS03}, every element of length $n$ in $G$ has order at most $Cn^{3/2}$ for some constant $C$. \Wlog $C = 2^\ell$ for some $\ell \in \N$. 
	On input of a power word with all periods of length at most $n$, we can compute the smallest $k$ with $2^k \geq n$ in \AC.
	We have $2^k \le 2n$. Now, we know that an element of length $n$ has order bounded by $2^{2k+\ell}$. Since the order of every element of $G$ is a power of two, this means that $g^{2^{2k+\ell}} = 1$ for all $g \in G$ of length at most $n$. Thus, we can reduce all exponents modulo $2^{2k+\ell}$ (\ie we drop all but the $2k+\ell$ least significant bits). Now all exponents are at most 
	$2^{2k+\ell} \leq 4Cn^2$ and the power word can be written as an ordinary word (to do this in \AC, we need a neutral letter to pad the output to a fixed word length). 
\end{proof}

\else
\begin{proof}[Proof sketch of \prettyref{thm:grigorchuck}.]
	By \cite[Theorem~6.6]{BartholdiGS03}, every element of length $N$ in  the Grigorchuk group has order at most $CN^{3/2}$ for some constant $C$. Since the order of every element is a power of two, we can reduce all exponents modulo the smallest power of two $\geq CN^{3/2}$ where $N$ is the length of the longest period $p_i$. After that the words are short an can be written without exponents.
\end{proof}
\fi

\section{Proof of \prettyref{thm:free_power_wp}}\label{sec:proof_free}

The proof of \prettyref{thm:free_power_wp} consists of two main steps: first we do some preprocessing leading to a particularly nice instance of the power word problem%
\ifFullVersion~(\prettyref{sec:preprocessing})\fi.
While this preprocessing is simple from a theoretical point of view, it is where the main part of the workload is performed during the execution of the algorithm. Then, in the second step, all exponents are reduced to polynomial size%
\ifFullVersion~(\prettyref{sec:shortening})\fi.
 After this shortening process, the power word problem can be solved by the ordinary word problem. 
The most difficult part is to prove correctness of the shortening process.
For this, we introduce a rewriting system over an extended alphabet of words with exponents%
\ifFullVersion~(\prettyref{sec:symbolic_reduction})\fi. 
\ifFullVersion 
The proof consists of a sequence of lemmas which all follow rather easily from the previous ones.
\else
We outline the proof in a sequence of lemmas which all follow rather easily from the previous ones and we give some small hints how to prove the lemmas.
\fi

\ifFullVersion
\subsection{Preprocessing}\label{sec:preprocessing}
\else
\subparagraph*{Preprocessing.}\label{sec:preprocessing}
\fi
We use the notations from the paragraph on free groups in Section~\ref{sec:prelims}. In particular, recall that $S= \set{a a^{-1} \to 1}{a \in \Sig}$.
Fix an arbitrary order on the input alphabet $\Sigma$. This gives us the lexicographic order on $\Sigma^*$, which is denoted by $\preceq$.
Let $\Omega \sse \IRR(S) \sse \Sigma^*$ 
denote the set of words $w$ such that 
\begin{itemize}
	\item $w$ is non-empty,
	\item $w$ is cyclically reduced (i.e, $w$ cannot be written as $a u a^{-1}$ for $a \in \Sigma$),
	\item $w$ is primitive (i.e, $w$ cannot be written as $u^n$ for $n \ge 2$),
	\item $w$ is lexicographically minimal among all cyclic permutations of $w$ and $w^{-1}$ (\ie $w \preceq uv$ for all $u,v \in \Sigma^*$ with $vu =w$ or $vu = w^{-1}$).
\end{itemize}
Notice that $\Omega$ consists of Lyndon words \cite[Chapter 5.1]{lot83} with the stronger requirement of being freely reduced, cyclically reduced and also minimal among the conjugacy class of the inverse. 
The first aim is to rewrite the input power word in the form
\begin{align} \label{eq:word-w}
\qquad\qquad w&=s_0p_1^{x_1}s_1 \cdots p_n^{x_n}s_n\qquad \text{with } p_i \in \Omega \text{ and } s_i \in \IRR(S).
\end{align}
The reason for this lies in the following crucial lemma which essentially says that, if a long factor of $p_i^{x_i}$ cancels with some $p_j^{x_j}$, then already $p_i= p_j$. Thus, only the same $p_i$ can cancel implying that we can make the exponents of the different $p_i$ independently smaller.

\begin{lemma}\label{lem:short_cancellation}
Let $p,q \in \Omega$, $x,y \in \mathbb{Z}$ and let $v$ be a factor of $p^x$ and $w$ a factor of  $q^{y}$. If $vw\RAS*S1$ and $\abs{v} = \abs{w} \geq \abs{p} + \abs{q} - 1$, then $p=q$. 
\end{lemma}	

\begin{proof}
Since $p$ and $q$ are cyclically reduced, $v$ and $w$ are freely reduced, i.e., $v=w^{-1}$ as words. Thus, $v$ has two periods $\abs{p}$ and $\abs{q}$. Since $v$ is long enough, by the theorem of Fine and Wilf \cite{FineWilf65} it has also a period of $\gcd(\abs{p},\abs{q})$. This means that also $p$ and $q$ have period $\gcd(\abs{p},\abs{q})$ (since cyclic permutations of $p$ and $q$ are factors of $v$). Assuming $\gcd(\abs{p},\abs{q}) < \abs{p}$, would mean that $p$ is a proper power contradicting the fact that $p$ is primitive. Hence, $\abs{p}=\abs{q}$. Since $\abs{v} \geq  \abs{p} + \abs{q}  - 1= 2\abs{p} -1$, $p$ is a factor of $v$, which itself is a factor of $q^{-y}$. Thus, $p$ is a cyclic permutation of $q$ or of $q^{-1}$. By the last condition on $\Omega$, this implies $p = q$. 
\end{proof}

\begin{lemma}\label{lem:preprocessing}
The following is in $\AC(\WP(F_2))$: given a power word $v$, compute a power word $w$ of the form \prettyref{eq:word-w} such that $v =_{F_X} w$.
\end{lemma}
\ifFullVersion	
\begin{proof}
	By \cite[Prop.\ 20]{Weiss16}, given a word $q \in \Sigma^*$, we can compute $\red{q}$ in $\AC(\WP(F_2))$. In order to transform the input $v = q_1^{y_1} \cdots q_n^{y_n}$ into the desired form \prettyref{eq:word-w}, we proceed as follows: 
	\begin{itemize}
		\item Freely reduce the $q_i$: for all $i$ compute $ \hat q_i = \red{q_i}$ (this is in  $\AC(\WP(F_2))$ 	by \cite{Weiss16}).
		\item Make all $ \hat q_i$ cyclically reduced: for all $i$ find the prefix $r_i \pref \hat q_i$ of maximal length such that  $\hat q_i = r_i\tilde q_ir_i^{-1}$ for some $\tilde q_i \in \Sigma^*$.  Now, $\tilde q_i$ is cyclically reduced and we have 
		\[ v =_{F_X} r_1{\tilde q_1}^{y_1} r_1^{-1} \cdots  r_n{\tilde q_n}^{y_n} r_n^{-1}.\]

		\item Make each $\tilde q_i$ primitive: find the shortest $\tilde p_i \pref \tilde q_i$ such that $q^{k_i}=\tilde q_i$ for some $1\leq k_i \leq \abs{\tilde q_i}$ and  replace $\tilde q_i$ by $\tilde p_i$ and $y_i$ by $x_i = k_iy_i$. The resulting word is
		\[ v =_{F_X} r_1{\tilde p_1}^{x_1} r_1^{-1}  \cdots r_n{\tilde p_n}^{x_n} r_n^{-1}.\]
		
		\item Make sure that $\tilde p_i$ is lexicographically minimal in its conjugacy class and the conjugacy class of $\tilde p_i^{-1}$. For all $i$ do the following: find $p_i$ such that $p_i$ is the lexicographically smallest element in $\set{uv}{vu= \tilde p_i \text{ or } vu =\tilde p_i^{-1}}$.
		If $p_i = u_iv_i$ for $v_iu_i= \tilde p_i$, replace $\tilde p_i^{x_i}$ by  $ u_i^{-1} p_i^{x_i} v_i$. Otherwise, we have $p_i = u_iv_i$ for $v_iu_i= \tilde p_i^{-1}$. Then we replace $\tilde p_i^{x_i}$ by  $ v_i  p_i^{-x_i} u_i^{-1}$.
        The words in between the $p_i$-powers are finally replaced by their freely reduced normal forms.
	\end{itemize}
Notice that each step does not destroy the conditions achieved in the previous steps. Hence, the resulting word is in $\Omega$. 
Moreover, the first step is in $\AC(\WP(F_2))$ by \cite{Weiss16}; all the other steps can be easily seen to be in \AC. 
For the last step observe that the words in between the $p_i$-powers are 
    concatenations of at most four freely reduced words. Thus, their freely reduced normal forms 
    can be computed in \AC.

Nevertheless, be aware that the \AC bounds for the second to last step only work in the presence of a neutral letter $\epsilon$ (\ie $\epsilon = 1$ in $F_X$). Otherwise, we get \TC bounds as we need to concatenate the words for all $i$. In any case, if the $q_i$ are already freely reduced, the whole procedure is in \TC.
\end{proof}
We call the steps performed in the proof of \prettyref{lem:preprocessing} the \emph{preprocessing steps}.
\else
The proof of this lemma is straightforward using \cite[Prop.\ 20]{Weiss16} in order to compute freely reduced words. We call these steps the \emph{preprocessing steps}.
\fi
Henceforth, we will assume that the inputs for the power word problem are given in the form \prettyref{eq:word-w}.

\ifFullVersion
\subsection{The symbolic reduction system}\label{sec:symbolic_reduction}
\else
\subparagraph*{The symbolic reduction system.}
\fi
\newcommand{\len}{\operatorname{\lambda}}

 We define an infinite alphabet $\Delta = \Delta' \cup \Delta''$ with $\Delta' = \Omega \times(\Z \setminus\smallset{0})$ and $\Delta'' = \IRR(S)\setminus\smallset{1}$.
 We write $p^x$ for $(p,x) \in \Delta'$. We can read every word over $\Delta$ as a word over $\Sigma$ in the natural 
way. Formally,
we can define a canonical projection $\pi: \Delta^* \to \Sigma^*$ that maps a symbol $a \in \Delta$ to the corresponding word over $\Sigma$,
but most of the times we will not write $\pi$ explicitly.

 Whenever there is the risk of confusion, we write $\abs{v}_\Sigma$ to denote the length of $v \in \Delta^*$ 
read over $\Sigma$
(i.e., $\abs{v}_\Sigma = |\pi(v)|$)
 whereas $\abs{v}_\Delta$ is the length over $\Delta$.  
 Moreover, we denote the number of letters from $\Delta'$ in $w$ with $\abs{w}_{\Delta'}$. 
 \ifFullVersion\else
Finally, for a symbol $s \in \Delta''$ define $\len(s) = \abs{s}_\Sigma$ and for $p^x \in \Delta'$ set $\len(p^x) = \abs{p}_\Sigma$. For 
 $u = a_1\cdots a_m \in \Delta^*$ with $a_i \in \Delta$ for $1 \leq i \le m$ we
 define $\len(u) = \sum_{i=1}^{m} \len(a_i)$.
 Thus, $\len(u)$ is the number of letters from $\Sigma$ required to write down $u$ ignoring the binary exponents. 
 \fi 
 
 A word $w$ as in \eqref{eq:word-w}, which has been preprocessed as in the previous section, can be viewed as word over $\Delta$ with
 $w \in ((\Delta''\cup\smallset{1})\Delta')^*(\Delta''\cup\smallset{1})$ and 
 $\abs{w}_{\Delta'} = n$ and $\abs{w}_{\Delta} \leq 2n + 1$ (we only have $\leq$ because some $s_i$ might be empty).

\newcommand{\sS}{T}
We define the infinite string rewriting system $T$ over $\Delta^*$ by the following rewrite rules, where 
$p^x, p^y, q^y \in \Delta'$, $s,t \in \Delta''$, $r \in \Delta'' \cup \smallset{1}$, and $d,e \in \Z$. 
Here, $p^0$ is identified with the empty word.
Note that the strings in the rewrite rules are over the alphabet 
$\Delta$, whereas the strings in the if-conditions are over the alphabet $\Sigma$. 
\begin{align}\allowdisplaybreaks
&&  p^xp^y 	&\to p^{x+y} 			&& \label{eq:rule_uu}\\
&&  p^xq^y 	&\to p^{x-d}rq^{y-e} 	&&\text{if }  p \neq q,  p^{x}q^{y} \RAS+S p^{x-d}rq^{y-e}  \in \IRR(S) \text{ for }\label{eq:rule_uv} \\
		\nonumber &&&&& 	r = p'q' \text{ with } p'\spref p^{\sign(x)} \text{ and } q'\ssuff q^{\sign(y)}  \\
&&  st 		&\to r 					&&\text{if } st \RAS+S r\in \IRR(S)\label{eq:rule_gh}\\
&&  p^xs 	&\to p^{x-d}r 			&&\text{if } p^xs \RAS+S p^{x-d}r\in \IRR(S)\text{ for }\label{eq:rule_ug}\\
		\nonumber &&&&&		r = p's' \text{ with } p'\spref p^{\sign(x)} \text{ and } s' \ssuff s  \\
&&  sp^x 	&\to rp^{x-d} 			&&\text{if }  sp^x \RAS+S rp^{x-d}\in \IRR(S)\text{ for }\label{eq:rule_gu}\\
		\nonumber &&&&& 	r =s' p' \text{ with }  s' \spref s \text{ and } p'\ssuff p^{\sign(x)} 
\end{align}

\begin{lemma}\label{lem:rewriting_length_bound}
The following length bounds hold in the above rules:
\begin{itemize}
\item in rule \prettyref{eq:rule_uv}: 
$0 < \abs{r}_\Sigma \leq \abs{p}_\Sigma+\abs{q}_\Sigma$,  
 $\abs{d} \leq \abs{q}_\Sigma $, and $\abs{e} \leq \abs{p}_\Sigma$
 \item in rules \prettyref{eq:rule_ug} and \prettyref{eq:rule_gu}:
 $\abs{d} \leq \abs{s}_\Sigma$.
\end{itemize}
\end{lemma}
\ifFullVersion
\begin{proof}
	The inequality  $\abs{r}_\Sigma \leq \abs{p}_\Sigma+\abs{q}_\Sigma$ for rule \prettyref{eq:rule_uv} holds, because $r$ is composed by a prefix of $p$ and a suffix of $q$. Moreover, $0 < \abs{r}$ because, if $p^d = q^{-e} $ for $d \neq 0 \neq e$, then, as $p,q \in \Omega$, we would have $p=q$, a contradiction.
	
	For the second inequality  $\abs{d} \leq \abs{q}_\Sigma$ assume that $\abs{d} > \abs{q}_\Sigma$. Then there is a suffix $t$ of $p^x$ with $\abs{t}_\Sigma \geq \abs{q}_\Sigma \abs{p}_\Sigma + 1$ which cancels with a prefix of $q^y$. As $\abs{q}_\Sigma \cdot \abs{p}_\Sigma + 1 \geq \abs{q}_\Sigma  + \abs{p}_\Sigma$ for $\abs{q}_\Sigma, \abs{p}_\Sigma \geq 1$ (which by assumption is the case), \prettyref{lem:short_cancellation} would imply that $p= q$, contradicting $p\neq q$. The inequality $\abs{e} \leq \abs{p}_\Sigma$
	can be shown analogously.
	
In rules \prettyref{eq:rule_ug} and \prettyref{eq:rule_gu}, $p^{-d}$ must be a prefix (resp.\ suffix) of $s$, which implies the bound.	
\end{proof}
\else
The inequalities  $\abs{d} \leq \abs{q}_\Sigma $ and $\abs{e} \leq \abs{p}_\Sigma$ follow from \prettyref{lem:short_cancellation}. The other inequalities are obvious. The next lemma is also straightforward from the definition.
\fi

\ifFullVersion

\begin{lemma}\label{lem:reduced_forms}
	Let $u,v \in \Delta^*$. We have 
	\begin{enumerate}
		\item $\pi(\IRR(\sS)) = \IRR(S)$,\label{eq:irrT}
		\item $u \RAS*T v$ implies $\pi(u) \RAS*S \pi(v)$, \label{eq:TS}
		\item $u=_{F_X}1$ if and only if $u \RAS{*}{\sS} 1$.\label{eq:T1}
	\end{enumerate} 
\end{lemma}
\begin{proof}
	The inclusion $\IRR(S) \sse \pi(\IRR(\sS))$ is trivial as $\IRR(S) = \Delta'' \cup \{1\} \subseteq \Delta \cup \{1\}$. For the other inclusion note that every two-letter factor can be made $S$-reduced using $\sS$. Point \ref{eq:TS} follows by a simple inspection of the rules. 
	
	For \ref{eq:T1}, first assume that $u \RAS{*}{\sS} 1$. Then by \ref{eq:TS} we have $\pi(u) \RAS{*}{S} 1$ and hence, $u=_{F_X}1$. On the other hand, if $u \RAS{*}{\sS} v \in \IRR(\sS)$ with $v \neq 1$, then  $\pi(u) \RAS*S \pi(v)$ with $\pi(v) \neq 1$. By \ref{eq:irrT} it follows that $\pi(v) \in \IRR(S)$. Hence, we have $u=_{F_X} v \neq_{F_X} 1$.
\end{proof}

\else

\begin{lemma}\label{lem:reduced_forms}
	For $u \in \Delta^*$ we have $u=_{F_X}1$ if and only if \smash{$u \RAS{*}{\sS} 1$}.
\end{lemma}

\fi

\ifFullVersion
For a symbol $s \in \Delta''$ define $\len(s) = \abs{s}_\Sigma$ and for $p^x \in \Delta'$ set $\len(p^x) = \abs{p}_\Sigma$. For 
$u = a_1\cdots a_m \in \Delta^*$ with $a_i \in \Delta$ for $1 \leq i \le m$ we
 define $\len(u) = \sum_{i=1}^{m} \len(a_i)$
and $\mu(u) = \max \{ \len(a_i) \mid 1 \le i \leq m \}$.
Thus, $\len(u)$ is the number of letters from $\Sigma$ required to write down $u$ ignoring the binary exponents. 

\begin{lemma}\label{lem:length_increase}
	Let $u \in \Delta^*$. 
	If $u \RAS{k}{\sS} v$, then $\len(v) \leq \len(u) +  2k\mu(u) \le (2k+1) \len(u)$. 
\end{lemma}

\begin{proof}
        Let $\Gamma \sse \Delta'$ be the set of all $p^x$ such that $p^y$ appears as a letter in $u$ for some $y \in \Z$. Then we have $v \in (\Gamma \cup \Delta'')^*$.
	Let $M = \max\{ \abs{p}_\Sigma \mid p^x \in \Gamma\}$. Clearly, we have $M \leq  \mu(u)$.
	Now we prove the lemma by induction on $k$. For $k=0$ we have $\len(v) = \len(u)$ and we are done.
	
	Rewriting rules \prettyref{eq:rule_uu} and \prettyref{eq:rule_gh} do not increase $\len(\;\!\cdot\;\!)$. 
	An application of rule \prettyref{eq:rule_ug} or \prettyref{eq:rule_gu} increases $\len(\;\!\cdot\;\!)$ by at most $M$ since $p^x$ is in $\Gamma$. Likewise, 
rule \prettyref{eq:rule_uv} increases $\len(\;\!\cdot\;\!)$ by at most $2M$. 
\end{proof}
\fi

\begin{lemma}\label{lem:rewriting_bound}
	Let $u \in \Delta^*$.
	If $u \RAS*{\sS}v$, then $u \RAS{\leq k}{\sS}v$ for $k = 2\abs{u}_\Delta + 4\abs{u}_{\Delta'} \leq 6\abs{u}_\Delta$. 
\end{lemma} 

\ifFullVersion
\begin{proof}
	Let $n = |u|_{\Delta'}$ and write $u = u_0 p_1^{x_1} u_1 \cdots p_n^{x_n} u_n$ with $u_i \in \Delta''^*$.
	During the rewriting process of $u$, for every pair of letters $p_i^{x_i}$ and $p_j^{x_j}$ in $u$ (for $i <j$), rule \prettyref{eq:rule_uv} can be applied at most once. Moreover, if $i < k < j < \ell$, then either a rule of type  \prettyref{eq:rule_uv} can be applied to  $p_i^{x_i}$ and  $p_j^{x_j}$ or to $p_k^{x_k}$ and  $p_\ell^{x_\ell}$ but not to both. This is because if the rule is applied to $p_i^{x_i}$ and  $p_j^{x_j}$, then everything in between has to cancel before~-- in particular $p_k^{x_k}$~--, so there is no way for it to be part of a reduction involving $p_k^{x_k}$ and  $p_\ell^{x_\ell}$. Let $P$ be the set of all pairs
	$(i,j)$ such that rule \prettyref{eq:rule_uv} is applied to $p_i^{x_i}$ and  $p_j^{x_j}$ in our reduction. From the above consideration, it follows
	that we obtain a forest with nodes in $P$, where $(k,l)$ is an ancestor of $(i,j)$ if the interval $[k,l]$ is contained $[i,j]$. The number of leaves
	of this forest is at most $n-1$. Hence, the number of nodes of the forest (and hence the number of pairs in $P$) is at most $2n-3$.
	This gives us at most $2n-3$ applications of rules of the form \prettyref{eq:rule_uv}. 
	Since \prettyref{eq:rule_uv} is the only rule that increases $\abs{\;\!\cdot\;\!}_\Delta$ (by one), it follows that 
	$\abs{\;\!\cdot\;\!}_\Delta$ can increase at most $2n-3$ times, and each time it increases by at most one.
		
	Rules  \prettyref{eq:rule_ug} and \prettyref{eq:rule_gu} either decrease $\abs{\;\!\cdot\;\!}_\Delta$ or reduce the number of non-reduced two-letter factors of the current word. 
	Since no application of any rule increases the number of non-reduced two-letter factors, there can be at most $\abs{u}_\Delta - 1$ application of non-length-decreasing applications of rules \prettyref{eq:rule_ug} and \prettyref{eq:rule_gu} ($\abs{u}_\Delta - 1$ is the number of factors of length two in the initial word).	
	
	The other rules \prettyref{eq:rule_uu} and \prettyref{eq:rule_gh} decrease $\abs{\;\!\cdot\;\!}_\Delta$.
	Since the initial length is $\abs{u}_\Delta $ and there are at most $2n-3$ applications of the only length-increasing rule \prettyref{eq:rule_uv}, at most $\abs{u}_\Delta + 2n-3$ applications of rules \prettyref{eq:rule_uu} and \prettyref{eq:rule_gh} and length-decreasing applications of rules \prettyref{eq:rule_ug} and \prettyref{eq:rule_gu} can occur. 

	Summing up we obtain $2n-3 + \abs{u}_\Delta - 1 + \abs{u}_\Delta + 2n-3 \leq 2\abs{u}_\Delta + 4\abs{u}_{\Delta'}  $ applications of rewriting rules.
\end{proof}
\else
\begin{proof}[Proof sketch]
The proof is based on the fact that at most $2|u|_{\Delta'}-3$ applications of rules of the form \prettyref{eq:rule_uv} can occur. These are the only length increasing rules. All other rules either decrease the number of non-reduced two-letter factors of $u$ (this can happen at most  $\abs{u}_\Delta - 1$ times) or decrease the length of $u$ (this can happen at most $\abs{u}_\Delta + 2|u|_{\Delta'}-3$ times). 
\end{proof}
\fi

Consider a word $u \in \Delta^*$ 
and $p \in \Omega$. Let $\Delta_p = \{ p^x \mid x \in \Z \setminus \{0\}\}$.
We can write $u$ uniquely as $u = u_0 p^{y_1} u_1 \cdots p^{y_m} u_m$ with 
$u_i \in (\Delta \setminus \Delta_p)^*$. We define $\eta_p^i(u)= \sum_{j=1}^{i}y_j$ and $\eta_p(u)=\eta_p^m(u)$.
\ifFullVersion
\begin{lemma}\label{lem:distance_easy}
	Let $u \RAS{}{\sS} v$ for $u,v \in \Delta^*$. 
Then for all $v' \in \Delta^*$ with
	$v' \pref v$ there is some $u' \in \Delta^*$ with $u' \pref u$ and
\[\abs{\eta_p(u') - \eta_p(v')}\leq \mu(u) \le \len(u).\]
	Moreover, if the applied rule is not of the form \prettyref{eq:rule_uu}, then  for all $i$ we have \[\abs{\eta_p^i(u) - \eta_p^i(v)} \leq \len(u).\]
\end{lemma}

\begin{proof}
	\Wlog we can assume that $v' \in \Delta^*\Delta_p$ (appending/deleting letters from $\Delta\setminus\Delta_p$ does not change $\eta_{p}(v')$). There are three possibilities: either the right hand side of the applied rule is only partially in $v'$ or it is entirely in $v'$ or entirely outside 
	$v'$. In the last case we can take $u'=v'$.
	
	In the second case, there is some $u' \in \Delta^*\Delta_p$ with
	$u' \pref u$ and
	$u' \RAS{}{\sS} v'$. Since all rules change $\eta_p(\cdot)$ by at most $\mu(\cdot)$ (by \prettyref{lem:rewriting_length_bound}), we have $\abs{\eta_p(u') - \eta_p(v')}\leq 
	\mu(u)$.
	
	In the third case, the right hand side of a rule of the form \prettyref{eq:rule_uv} or \prettyref{eq:rule_ug} overlaps $v'$.
	Let us write $v' = v'' p^y$. 
	This means that there exist
	$p^x \in \Delta$ and words $\alpha, \beta \in \Delta^{\leq 2}$ such that
	$p^x \alpha \to p^{y} \beta$ is a rule of $T$, $v'' p^{y} \beta$ is a prefix of $v$ and 
	$v'' p^{x} \alpha$ is a prefix of $u$. Let $u' = v'' p^x$.
	Since $|x-y| \le \mu(u)$ (by \prettyref{lem:rewriting_length_bound}), we have $\abs{\eta_p(u') - \eta_p(v')}\leq \mu(u)$.

In the case that the applied rule is not of the form \prettyref{eq:rule_uu}, then every $p$-power in $v'$ corresponds to a unique $p$-power in $u'$ and vice-versa. Thus, the second statement follows.
\end{proof}

\else
By \prettyref{lem:rewriting_length_bound} we know that all rules of $T$ change $\eta_p(\cdot)$ by at most $\len(u)$. We can use this observation in order to show the next lemma by induction on $k$.
\fi

\begin{lemma}\label{lem:distance}
	Let $u \RAS{k}{\sS}v$. Then for all $v'\pref v$ with $v' \in \Delta^*$ there is some $u'\in \Delta^*$ with $u' \pref u$ and 
\ifFullVersion
	\[\abs{\eta_p(u') - \eta_p(v')}\leq (k+1)^2\len(u).\]
\else
	$\abs{\eta_p(u') - \eta_p(v')}\leq (k+1)^2\len(u).$
\fi
\end{lemma}

\ifFullVersion
\begin{proof}
	We proceed by induction on $k$. For $k=0$ the statement is trivial. Let $v' \pref v$ and $u \RAS{k - 1}{\sS} w \RAS{}{\sS} v$. By \prettyref{lem:length_increase}, $\len(w) \leq (2k-1) \len(u)$. Hence, by \prettyref{lem:distance_easy},  there is some $w'\pref w$ with $\abs{\eta_p(v') - \eta_p(w')}\leq (2k-1) \len(u)$. By induction, we know that there is some $u'\in \Delta^*$ with $u' \pref u$ and  $\abs{\eta_p(u') - \eta_p(w')}\leq k^2\len(u)$. Hence, \[\abs{\eta_p(u') - \eta_p(v')}\leq (k^2 + 2k-1)\len(u) \le (k+1)^2\len(u).\qedhere\] 
\end{proof}
\fi

\ifFullVersion
\subsection{The shortened version of a word}\label{sec:shortening}
\else
\subparagraph*{The shortened version of a word.}
\fi

Take a word $u \in \Delta^*$ and $p \in \Omega$ and write $u$ as
$u = u_0 p^{y_1} u_1 \cdots p^{y_m} u_m$ with $u_i \in (\Delta \setminus \Delta_p)^*$ (we are only interested in the case that $p^x$ appears as a letter in $u$ for some $x \in\Z$).
Let $\cC$ be a finite set of finite, non-empty, non-overlapping intervals of integers, \ie we can write 
$\cC = \set{[\ell_j, r_j]}{1 \le j \le k}$ for $k = \abs{\cC}$ and $\ell_j \leq r_j$ for all $j$. We can assume that the intervals are ordered increasingly, \ie we have $r_j < \ell_{j+1}$.
We set $d_j = r_j - \ell_j + 1>0$. 
We say that $u$ is \emph{compatible} with $\cC$ if $\eta_p^i(u) \not\in [\ell_j,r_j]$ for any $i,j$. If $w$ is compatible with $\cC$, we define the \emph{shortened version} $\cS_\cC(u)$ of $u$: for $i \in \smallset{1, \dots, m}$ we set 
\[
C_i = C_i(u) = \begin{cases}
\set{j}{1 \le j \le k, \eta_p^{i-1}(u) < \ell_j\leq r_j < \eta_p^i(u)}  &\text{if } y_i >0\\
\set{j}{1 \le j \le k, \eta_p^i(u) < \ell_j \leq r_j < \eta_p^{i-1}(u)}  &\text{if } y_i <0,
\end{cases}\]
\ie $C_i$ collects all intervals between $\eta_p^{i-1}(u) $ and $ \eta_p^i(u)$. Then $\cS_\cC(u)$ is defined by
%
\begin{align*}
\cS_\cC(u) &= u_0 p^{z_1} u_1 \cdots p^{z_m} u_m \qquad \text{where}\\
z_i &= y_i - \sign(y_i)\cdot\sum_{j\in C_i}d_j =  \begin{cases}
y_i - \sum_{j\in C_i}d_j& \text{if } y_i > 0,\\
y_i + \sum_{j\in C_i}d_j& \text{if } y_i <0. 
\end{cases}
\end{align*}
\ifFullVersion\else
A straightforward computation yields the next lemma: 
\fi
\begin{lemma}\label{lem:z_not_zero}
	For all $i$ we have $z_i\neq 0$ and $\sign(z_i) = \sign(y_i)$. In particular, if $u\in \IRR(T)$, then also $\cS_\cC(u) \in \IRR(T)$.
\end{lemma}%
\ifFullVersion
\begin{proof}
Assume that $y_i > 0$ (the other case is completely symmetric) and let $C_i = [\alpha, \beta]$ (as $\cC$ is ordered increasingly, we know that $C_i$ is an interval). Then 
	\begin{align*}
	\sum_{j\in C_i}d_j &=\sum_{j=\alpha}^{\beta} (r_j - \ell_j +1) \\
	&\leq r_\alpha - \ell_\alpha +1 + \sum_{j=\alpha + 1}^{\beta} (r_j - r_{j-1}) \tag{since $\ell_j > r_{j-1}$}\\
	&\leq r_\alpha - \ell_\alpha +1 + r_\beta - r_\alpha \\
	&= r_\beta - \ell_\alpha +1\\ 
	&\leq \eta_p^{i}(u) - \eta_p^{i-1}(u) - 1 = y_i -1.\tag{since $\alpha,\beta \in C_i$}
	\end{align*}
	Thus, $z_i = y_i - \sign(y_i)\cdot\sum_{j\in C_i}d_j \geq 1$. This implies the lemma.
\end{proof}
\fi
\newcommand{\dist}{\operatorname{dist}}%
Furthermore, we define 
\ifFullVersion
\begin{align*}
	\dist_p(u,\cC) = \min\set{\abs{\eta_p^i(u) - x} }{0 \leq i \leq m, x \in [\ell,r] \in \cC}.
\end{align*}
\else
$\dist_p(u,\cC) = \min\set{\abs{\eta_p^i(u) - x} }{0 \leq i \leq m, x \in [\ell,r] \in \cC}$.
\fi
Note that $\dist_p(u,\cC) > 0$ if and only if $u$ is compatible with $\cC$. Moreover, if $\dist_p(u,\cC) = a$,
$v= v_0 p^{z_1} v_1 \cdots p^{z_{m}} v_{m}$, and $\abs{\eta_p^i(u) - \eta_p^i(v)} \leq b$ for all $i \leq m$, then $\dist_p(v,\cC) \geq a - b$.

\ifFullVersion
\begin{lemma}\label{lem:shorten_one_step}
	If  $\dist_p(u,\cC) > \len(u)$ and $u\RAS{}{\sS} v$, then $S_{\cC}(u) \RAS{}{\sS} S_{\cC}(v)$.
\end{lemma}

\begin{proof}
Notice that, in particular, $u$ is compatible with $\cC$. Moreover, \prettyref{lem:distance_easy} implies that $\dist_p(v,\cC) > 0$; so also $v$ is compatible with $\cC$. Let $u=u_0 p^{y_1} u_1 \cdots p^{y_m} u_m$ and $\cS_\cC(u) = u_0 p^{z_1} u_1 \cdots p^{z_m} u_m$.
	
	We distinguish which of the rules of $\sS$ is applied.
	When applying one of the rules \prettyref{eq:rule_uv}--\prettyref{eq:rule_gu}, we have $C_i(u) = C_i(v)$ for all $i$ because by \prettyref{lem:distance_easy} we have $\abs{\eta_p^i(u) - \eta_p^i(v)} \leq \len(u)$. Thus, $\dist_p(v,\cC) \geq1$ and the shortening process does the same on $v$ as on $u$. This means that we can apply the same rule of $T$ to $S_{\cC}(u)$ obtaining $S_{\cC}(u) \RAS{}{\sS} S_{\cC}(v)$.
	
	 Consider now the case that the applied rule is of the form \prettyref{eq:rule_uu}, \ie $p^{y_i}p^{y_{i+1}} \to p^{y_i + y_{i+1}}$ and we have 
	\begin{align*}
	u=u_0 p^{y_1} u_1 \cdots p^{y_{i-1}} u_{i-1} p^{y_i} p^{ y_{i+1}}u_{i+1} p^{y_{i+2}} u_{i+2} \cdots p^{y_m} u_m,\\
	v=u_0 p^{y_1} u_1 \cdots p^{y_{i-1}} u_{i-1} p^{y_i + y_{i+1}}u_{i+1} p^{y_{i+2}} u_{i+2} \cdots p^{y_m} u_m. 
	\end{align*}	
	First assume that $y_i + y_{i+1}=0$, i.e., 
	$v = u_0 p^{y_1} u_1 \cdots p^{y_{i-1}} u_{i-1}u_{i+1} p^{y_{i+2}} u_{i+2}  \cdots p^{y_m} u_m$.
	We have $C_j(v) = C_j(u)$ for $j < i$ and $C_j(v) = C_{j+2}(u)$ for $j \ge i$.
	We obtain
	\[
	\cS_\cC(v)= u_0 p^{z_1} u_1 \cdots p^{z_{i-1}} u_{i-1}u_{i+1} p^{z_{i+2}} u_{i+2}  \cdots p^{z_m} u_m.
	\]
	Moreover, $C_i(u)=C_{i+1}(u)$, which yields $z_i = -z_{i+1}$, i.e., $z_i + z_{i+1}=0$.
	This yields $\cS_\cC(u) \RAS{}{\sS} S_{\cC}(v)$.

	Now assume that $y_i + y_{i+1} \neq 0$.
	We have $C_j(v) = C_j(u)$ for $j < i$ and $C_j(v) = C_{j+1}(u)$ for $j > i$. Hence, 
	\[ \cS_\cC(v)=u_0 p^{z_1} u_1 \cdots p^{z_{i-1}} u_{i-1} p^{\tilde z_i} u_{i+1} p^{z_{i+2}} u_{i+2} \cdots p^{z_m} u_m,\] where $\tilde z_i = y_i+ y_{i+1} - \sign(y_i+y_{i+1})\cdot\sum_{j\in C_i(v)}d_j$.
	
	For $C_i(v)$ there are two possibilities: either $y_i$ and $y_{i+1}$ have the same sign or opposite signs. If they have the same sign, then $C_i(v) = C_i(u) \cup C_{i+1}(u)$ and we obtain
	\begin{align*}
	\tilde z_i  &=y_i+ y_{i+1} - \sign(y_i+y_{i+1})\cdot\!\!\!\sum_{j\in C_i(v)}d_j\\
	&= y_i - \sign(y_i)\cdot\!\!\!\sum_{j\in C_i(u)}d_j + y_{i+1} - \sign(y_{i+1})\cdot\!\!\!\!\!\sum_{j\in C_{i+1}(u)}d_j\\
	&= z_i + z_{i+1}.
	\end{align*}
	Hence, $S_{\cC}(u) \RAS{}{\sS} S_{\cC}(v)$.
	Now assume that $y_i$ and $y_{i+1}$ have opposite signs and that $\abs{y_i} > \abs{y_{i+1}}$ and $y_i > 0$ (the other cases are symmetric). Then $C_i(v) = C_i(u) \setminus C_{i+1}(u)$. Thus, 
	\begin{align*}
	\tilde z_i  &=y_i+ y_{i+1} - \sign(y_i+y_{i+1})\cdot\!\!\!\sum_{j\in C_i(v)}d_j\\
	&=y_i+ y_{i+1} - \sign(y_i)\cdot\!\!\!\!\!\!\!\!\!\sum_{j\in  C_i(u) \setminus C_{i+1}(u)}d_j\\
	&= y_i - \sign(y_i)\cdot\!\!\!\sum_{j\in C_i(u)}d_j + y_{i+1} - \sign(y_{i+1})\cdot\!\!\!\!\!\!\sum_{j\in C_{i+1}(u)}d_j\\
	&= z_i + z_{i+1}.
	\end{align*}
	Hence, also in this case, we obtain $S_{\cC}(u) \RAS{}{\sS} S_{\cC}(v)$.
\end{proof}
\fi

\begin{lemma}\label{lem:shorten}
	If $\dist_p(u,\cC) > (k+1)^2\len(u)$ and $u\RAS{k}{\sS} v$, then $S_{\cC}(u) \RAS{k}{\sS} S_{\cC}(v)$.
\end{lemma}
\ifFullVersion
\begin{proof}
We prove the lemma by induction on $k$. The case $k=0$ is trivial.
Now assume that $\dist_p(u,\cC) > (k+1)^2\len(u)$ and $u \RAS{k-1}{\sS} w \RAS{}{\sS} v$. 
Induction yields $S_{\cC}(u) \RAS{k-1}{\sS} S_{\cC}(w)$.
By \prettyref{lem:distance},
every number $\eta_{p}^i(w)$ has distance at most $k^2\len(u)$ from a number $\eta^j_{p}(u)$ (for some $j$).
Hence, we have 
\[
\dist_p(w,\cC) \geq \dist_p(u,\cC) - k^2\len(u) > ( (k+1)^2 -  k^2 )\len(u) > \len(u).
\]
 Thus, \prettyref{lem:shorten_one_step} implies $S_{\cC}(w) \RAS{}{\sS} S_{\cC}(v)$ and hence, $S_{\cC}(u) \RAS{k}{\sS} S_{\cC}(v)$.
\end{proof}
\else
\begin{proof}[Proof sketch]
The first step for proving this lemma is to show that if  $\dist_p(u,\cC) > \len(u)$ and $u\RAS{}{\sS} v$, then already $S_{\cC}(u) \RAS{}{\sS} S_{\cC}(v)$. To see this this, we distinguish between the rules applied: 	When applying one of the rules \prettyref{eq:rule_uv}--\prettyref{eq:rule_gu}, we have $C_i(u) = C_i(v)$ for all $i$ since the exponents are only changed slightly. Thus, the shortening process does the same on $v$ as on $u$. When applying a rule \prettyref{eq:rule_uu}, the exponents are added, which is compatible with the shortening process.
Now we obtain the lemma by induction on $k$. In order to see that  $\dist_p(u,\cC) > \len(u)$ is satisfied in the inductive step, we use \prettyref{lem:distance}.
\end{proof}
\fi
\begin{figure}[t]
\pgfkeys{/pgf/inner sep=0em}  
  \centering
    \begin{tikzpicture}
    \node[circle,fill,minimum size=1mm,label=left:{$c_5$}] (0) {} ;
    \node[circle,fill,minimum size=1mm, above right = .5 and .5/2 of 0] (1) {} ;
    \path let \p1 = (1) in node[label=left:{$c_6$}] at (0,\y1) (y1) {};
    \node[circle,fill,minimum size=1mm, below right = 3.5 and 3.5/2 of 1] (2) {} ;
    \path let \p1 = (2) in node[label=left:{$c_2$}] at (0,\y1) (y2) {};
    \node[circle,fill,minimum size=1mm, above right = .5 and .5/2 of 2] (3) {} ;
    \path let \p1 = (3) in node[label=left:{$c_3$}] at (0,\y1) (y3) {};
    \node[circle,fill,minimum size=1mm, above right = 5 and 5/2 of 3] (4) {} ;
    \path let \p1 = (4) in node[label=left:{$c_9$}] at (0,\y1) (y4) {};
    \node[circle,fill,minimum size=1mm, below right = 1.5 and 1.5/2 of 4] (5) {} ;
    \path let \p1 = (5) in node[label=left:{$c_7$}] at (0,\y1) (y5) {};
    \node[circle,fill,minimum size=1mm, above right = 1 and 1/2 of 5] (6) {} ;
    \path let \p1 = (6) in node[label=left:{$c_8$}] at (0,\y1) (y6) {};
    \node[circle,fill,minimum size=1mm, below right = 6 and 6/2 of 6] (7) {} ;
    \path let \p1 = (7) in node[label=left:{$c_1$}] at (0,\y1) (y7) {};
    \node[circle,fill,minimum size=1mm, above right = 2.5 and 2.5/2 of 7] (8) {} ;
    \path let \p1 = (8) in node[label=left:{$c_4$}] at (0,\y1) (y8) {};
      
    \fill [red!30] ($(y7)+(0,.3)$) rectangle ($(y2)+(11,-.3)$);
    \fill [red!30] ($(y3)+(0,.3)$) rectangle ($(y8)+(11,-.3)$);
    \fill [red!30] ($(y8)+(0,.3)$) rectangle ($(0)+(11,-.3)$);
    \fill [red!30] ($(y5)+(0,.3)$) rectangle ($(y6)+(11,-.3)$);
    
    \draw[->] (0) -- ($(0)+(11,0)$);
    \draw[->]  ($(0)+(0,-4.5)$) -- ($(0)+(0,3)$);
    \draw (0) -- (1) -- (2) -- (3) -- (4) -- (5) -- (6) -- (7) -- (8);
    \draw[help lines] (y1) -- ($(y1)+(11,0)$);
    \draw[help lines] (y2) -- ($(y2)+(11,0)$);
    \draw[help lines] (y3) -- ($(y3)+(11,0)$);
    \draw[help lines] (y4) -- ($(y4)+(11,0)$);
    \draw[help lines] (y5) -- ($(y5)+(11,0)$);
    \draw[help lines] (y6) -- ($(y6)+(11,0)$);
    \draw[help lines] (y7) -- ($(y7)+(11,0)$);
    \draw[help lines] (y8) -- ($(y8)+(11,0)$);

   \draw [decorate,decoration={brace,mirror,amplitude=3pt}]  
   ($(y7)+(11,.3)$) -- ($(y2)+(11,-.3)$) node [midway,xshift=0.3cm] {$d_1$};
   
   \draw [decorate,decoration={brace,mirror,amplitude=3pt}]  
   ($(y3)+(11,.3)$) -- ($(y8)+(11,-.3)$) node [midway,xshift=0.3cm] {$d_3$};
   
   \draw [decorate,decoration={brace,mirror,amplitude=3pt}]  
   ($(y8)+(11,.3)$) -- ($(0)+(11,-.3)$) node [midway,xshift=0.3cm] {$d_4$};
   
   \draw [decorate,decoration={brace,mirror,amplitude=3pt}]  
   ($(y5)+(11,.3)$) -- ($(y6)+(11,-.3)$) node [midway,xshift=0.3cm] {$d_7$};
    \end{tikzpicture}
    \caption{\label{fig-shortening}The red shaded parts represent the intervals from the set
    $\cC_{u,p}^{K}$ in \eqref{def-C^k}.
    The differences $c_3-c_2$, $c_6-c_5$, $c_7-c_6$ and $c_9-c_8$ are strictly smaller than $2K$.}
\end{figure}
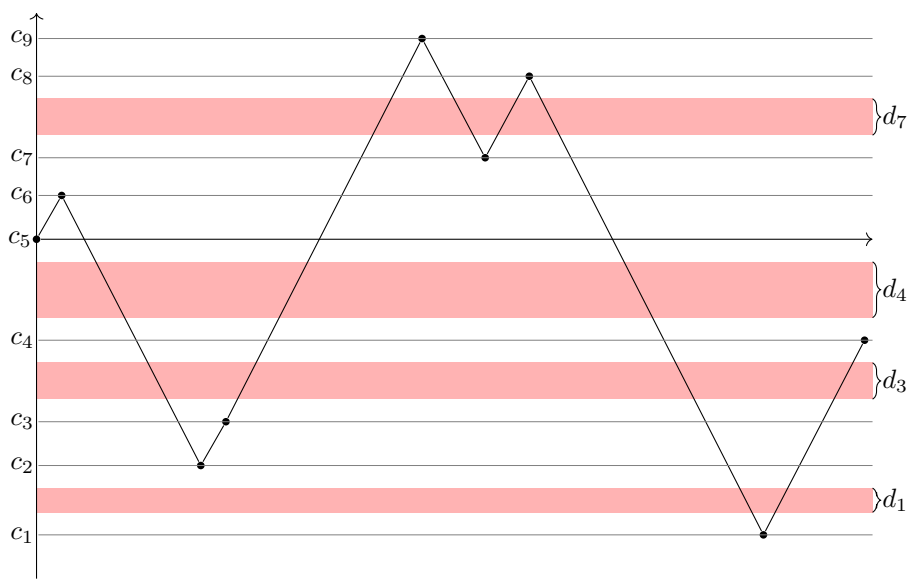

We define a set of intervals which should be ``cut out''  from $u$ as follows:
We write $\oneset{c_1, \dots, c_l} = \set{\eta_p^i(u)}{0 \le i \leq m}$ with $c_1 < \dots< c_l$ and we set 
\begin{equation} \label{def-C^k}
\cC_{u,p}^{K} = \set{[c_j + K, c_{j+1}-K]}{1\leq j \leq l - 1, c_{j+1} - c_j \geq 2K}.
\end{equation}
Notice that  $\dist_p(u,\cC_{u,p}^{K}) = K$ (given that $\cC_{u,p}^{K} \neq \emptyset$).
The situation is shown in Figure~\ref{fig-shortening}.

\begin{proposition}\label{prop:main_lemma}
	Let  $p\in \Omega$, $u = u_0 p^{y_1} u_1 \cdots p^{y_m} u_m \in \Delta^*$ with $u_i \in (\Delta \setminus \Delta_p)^*$,
	and $K = (6\abs{u}_\Delta+1)^2\len(u)+1$. Then $u=_{F_X}1$ if and only if $S_{\cC}(u)=_{F_X} 1$  for $\cC = \cC_{u,p}^{K} $.
\end{proposition}

\begin{proof}
 By \prettyref{lem:reduced_forms} we have $u=_{F_X}1$ if and only if $u \RAS{*}{\sS} 1$. Let $k = 6\abs{u}_\Delta$.
 By \prettyref{lem:rewriting_bound}, for all $u\RAS{*}{\sS} v$ we have $u\RAS{k}{\sS} v$. By the choice of $\cC$, we have $\dist_p(u,\cC) > (k+1)^2\len(u)$. Hence, we can apply \prettyref{lem:shorten}, which implies that $S_{\cC}(u) \RAS{*}{\sS} S_{\cC}(v)$ where $v$ is a $T$-reduced (thus freely reduced) word for $u$. Clearly, if $v$ is the empty word, then $S_{\cC}(v)$ will be the empty word. On the other hand, if $v$ is non-empty, by \prettyref{lem:z_not_zero}, $S_{\cC}(v)$ is non-empty and $T$-reduced. Hence, we have $u=_{F_X}1$ if and only if $S_{\cC}(u)=_{F_X}1$. 
\end{proof}

\begin{lemma}\label{lem:exponents_small}
Let  $p$, $u$, $K$, and $\cC$ be as in \prettyref{prop:main_lemma} and $S_{\cC}(u)= u_0 p^{z_1} u_1 \cdots p^{z_m} u_m$.
Then $\abs{z_i} \leq m\cdot( 2 \cdot(6\abs{u}_\Delta+1)^2\cdot\len(u)+1)$ for all $1 \le i \leq m$.
\end{lemma}

\ifFullVersion
\begin{proof}
Let $K = (6\abs{u}_\Delta+1)^2\len(u)+1$ and let $c_1, \dots, c_l$ as above (note that $l \leq m +1$). We have
	\begin{align*}
	\abs{z_i} &= \Big| y_i - \sign(y_i)\cdot\sum_{j\in C_i}d_j\Big|\\
			  &= \Big| y_i - \sign(y_i)\cdot\sum_j \max\{0, c_{j+1} - c_j - 2K+1\} \Big|\\
			  &\leq  \Big| y_i - \sign(y_i)\cdot\sum_j  (c_{j+1} - c_j) \Big| + (l-1) (2K-1) = (l-1) (2K-1)\leq m (2K-1).
	\end{align*} 
	Here the sums in the second and the last line range over all $j$ such that $\eta_{p}^i(w)\leq c_j < \eta_{p}^{i+1}(w)$ (resp.\ $\eta_{p}^{i+1}(w)\leq c_j < \eta_{p}^i(w)$ if $y_i < 0$).
\end{proof}
\fi

\begin{proof}[Proof of \prettyref{thm:free_power_wp}]
	We start with the preprocessing as described in \prettyref{lem:preprocessing} leading to a word $w=s_0p_1^{x_1}s_1 \cdots p_n^{x_n}s_n$ with $p_i \in \Omega $ and $ s_i \in \IRR(S)$ as in \prettyref{eq:word-w}. After that we apply the shortening procedure for all $p \in \set{p_i}{1\leq i \leq n}$. This can be done in parallel for all $p$, as the outcome of the shortening only depends on the $p$-exponents. By \prettyref{lem:exponents_small} this leads to a word $\hat w$ of polynomial length. Finally, we can test whether  $\hat w =_{F_X} 1$ using one oracle gate to the word problem for $F_2$ (recall that $F_2$ contains a copy of $F_X$). The computations for shortening only involve iterated addition (and comparisons of integers), which is in \TC and, thus, can be solved in $\AC$ with oracle gates for the word problem for $F_2$.
\end{proof}
\ifFullVersion
\section{Proof of \prettyref{thm:tc0-wreath}}\label{sec:proof-tc0-wreath}
\else
\section{The power word problem in wreath products}\label{sec:proof-wreath}
\fi

The goal of this section is to prove Theorems~\ref{thm:tc0-wreath} and \ref{thm:wreath-coNP}.
We first fix some notation. We fix a finitely generated group $G$ with the finite  symmetric generating set
$\Sigma$. For $\Z$ we fix the generator $a$. Hence $\Sigma \cup \{a,a^{-1}\}$
is a symmetric generating set for the wreath product $G \wr \Z$.
For a word $w = v_0 a^{e_1} v_1 \cdots a^{e_n} v_n$
with $e_i \in \{-1,1\}$ and $v_i \in \Sigma^*$ let $\sigma(w) = e_1  + \cdots + e_n$ be the integer represented
by $w$. 
\ifFullVersion
	For a word $w \in (\Sigma \cup \{a,a^{-1}\})^*$ let $\pi_a(w)$ be the projection on the subalphabet $\{a,a^{-1}\}$
	and define $\sigma(w) = \sigma(\pi_a(w))$ (the $\Z$-shift of $w$).
\fi
Moreover, we  denote with $I(w)$ the interval $[b,c] \subseteq \Z$, where $b$ (resp., $c$) is the minimal
(resp., maximal) integer of the form $e_1 + \cdots + e_i$ for $0 \le i \le n$.
Note that if $w$ represents $(f,d) \in G \wr \Z$, then $d = \sigma(w)$,
$\supp(f) \subseteq I(w)$ and $0,d \in I(w)$.
\ifFullVersion
	 For an integer interval $[a,b] \subseteq \mathbb{Z}$ and 
	$z \in \mathbb{Z}$ we write $z + [a,b]$ for the interval $[z+a,z+b]$.
	A function $f : \Z \to G$ is called {\em periodic with period} $q \geq 1$ on the interval $[b,c]$ if $ f(x) = f(x+q)$ for all $b \leq x \leq c-q$.
\fi

\ifFullVersion
\subsection{Periodic words over groups} \label{sec-periodic}
\else
\subparagraph*{Periodic words over groups.} \label{sec-periodic}
\fi

We  recall a construction from \cite{GanardiKLZ18}.
With $G^+$ we denote the set of all tuples $(g_0,\ldots, g_{q-1})$ over $G$ of arbitrary length $q \geq 1$.
With $G^\omega$ we denote the set of all mappings $f : \mathbb{N} \to G$. Elements of $G^\omega$ 
can be seen as infinite sequences (or words) over the set $G$.
We define the binary operation $\otimes$ on $G^\omega$  by pointwise 
multiplication: $(f \otimes g)(n) = f(n) g(n)$. 
\ifFullVersion
In fact, $G^\omega$ together with the multiplication $\otimes$ is the direct product of $\aleph_0$ many copies
of $G$.
\fi The identity element is the mapping $\id $ with $\id(n)=1$ for all $n \in \mathbb{N}$.
For $f_1, f_2, \ldots, f_n \in G^\omega$ we write $\bigotimes_{i=1}^n f_i$ for $f_1 \otimes f_2 \otimes \cdots \otimes f_n$.
If $G$ is abelian, we write  $\sum_{i=1}^n f_i$ for  $\bigotimes_{i=1}^n f_i$.
A function $f \in G^\omega$ is periodic with period $q \geq 1$ if $f(k) = f(k+q)$ for all $k \geq 0$.
\ifFullVersion
Note that in this situation, $f$ might be periodic with a smaller period $q' < q$.
\fi
Of course, a periodic function $f$ with period $q$ can specified by the tuple $(f(0), \ldots, f(q-1))$.
Vice versa, a tuple $u = (g_0, \ldots, g_{q-1}) \in G^+$ defines the periodic function $f_u \in G^\omega$ with
$f_u(n \cdot q + r) = g_{r}$ for $n \geq 0$ and $0 \leq r < q$.
One can view this mapping as the sequence $u^\omega$ obtained by taking infinitely many repetitions of $u$.
Let $G^\rho$ be the set of all periodic functions from $G^\omega$.
If $f_1$ is periodic with period $q_1$ and $f_2$ is periodic with period $q_2$, then 
$f_1 \otimes f_2$ is periodic with period $q_1 q_2$ (in fact, $\mathrm{lcm}(q_1, q_2)$). Hence, $G^\rho$
forms a countable subgroup of $G^\omega$. Note that  $G^\rho$ is not finitely generated: The subgroup
generated by elements $f_i \in G^\rho$ with period $q_i$ ($1 \leq i \leq n$) contains only functions
with period $\mathrm{lcm}(q_1, \ldots, q_n)$.
For $n \geq 0$ we define the subgroup $G^\rho_n$ of all $f \in G^\rho$ with $f(k) = 1$ for all $0 \leq k \leq n-1$.
We consider the uniform membership problem for subgroups $G^\rho_n$, $\MEM(G^\rho_{\ast})$ for short:
\begin{itemize}
\item input: tuples $u_1, \ldots, u_n \in G^+$ (elements of $G$ are represented by finite words over $\Sigma$) and a binary encoded number $m$.
\item question: does $\bigotimes_{i=1}^n f_{u_i}$ belong to $G^\rho_m$?
\end{itemize}
The following result was shown in \cite{GanardiKLZ18}:

\begin{theorem} \label{thm:abelian-membership}
For every finitely generated abelian group $G$, $\MEM(G^\rho_{\ast})$ \ifFullVersion belongs to \else is in \fi $\TC$.
\end{theorem}

\ifFullVersion
\subsection{Periodic words arising from powers in $G \wr \Z$}
The relationship between periodic functions on integer intervals and powers in wreath products is 
expressed by the following lemma. Note that the interval $[b+s, c-s]$ might be empty, in which case the conclusion of the lemma
is trivially true. 
\fi

\begin{lemma} \label{lem:periodic-wreath}
Let $w \in  (\Sigma \cup \{a,a^{-1}\})^*$ with $\sigma(w) \neq 0$, $n \geq 1$, and $I(w^n) = [b,c]$.
Moreover, let $s = c-b+1$ be the size of the interval $I(w)$ and let $(g, n \cdot \sigma(w)) \in G \wr \Z$ be the 
group element represented by $w^n$. Then $g$ is periodic on the interval $[b+s, c-s]$
with period $|\sigma(w)|$.
\end{lemma}

\ifFullVersion
\begin{proof}
Let us assume that $\sigma(w) > 0$; the case $\sigma(w) < 0$ is symmetric.
Let $(f,\sigma(w)) \in G \wr \Z$ be the group element represented by $w$.
Consider a position $k \in [b+s,c-s]$ and let $I_k = \{ i \in \mathbb{Z} \mid k \in i \cdot \sigma(w) + I(w)\}$. For all $i \in \mathbb{Z}$, if 
$k \in i \cdot \sigma(w) + I(w)$, then we must have $0 \le i \le n-1$.
 Thus, $I_k$ is a subinterval of $[0,n-1]$.
Let $l_k = \min(I_k)$ and $p_k \in I(w)$ such that $k = l_k \cdot \sigma(w) + p_k$. Then
$k = i \cdot \sigma(w) + p_k - (i-l_k) \cdot \sigma(w)$ and $p_k - (i-l_k) \cdot \sigma(w) \in I(w)$
 for all $i \in I_k$. For the  function value $g(k)$ we then obtain
 \[
 g(k) = \prod_{i \in I_k} f(p_k - (i-l_k) \cdot \sigma(w)) .
 \]
 Note that this value is uniquely determined by $p_k$ and the size of $I_k$.
 
Now assume additionally that $b+s \leq k < k + \sigma(w)  \le c-s$. We then have for all $i \in \mathbb{Z}$:
$i \in I_{k+\sigma(w)}$ if and only if 
$k + \sigma(w) \in i \cdot \sigma(w) + I(w)$ if and only if 
$k  \in (i-1) \cdot \sigma(w) + I(w)$ if and only if 
$i-1 \in I_k$ if and only if $i \in I_k+1$. Clearly $I_k$ and $I_k+1$ have the same size.
Moreover, $l_{k+\sigma(w)} = l_k+1$. Hence, we have 
\[ p_{k+\sigma(w)} = k+\sigma(w) - l_{k+\sigma(w)} \cdot \sigma(w) = 
k+\sigma(w) - (l_k+1) \cdot \sigma(w) = k - l_k \cdot \sigma(w) = p_k .
\]
We thus obtain $g(k+\sigma(w)) = g(k)$.
This concludes the proof of the lemma. 
\end{proof}
Here is an example for the situation from Lemma~\ref{lem:periodic-wreath}
\fi

\begin{example}
Let us consider the wreath product $\Z \wr \Z$ and let the left copy of $\Z$ in the wreath product
be generated by $b$. 
Consider the word $w = b a^{-1} b a b a b^3 a b^3 a b^5 a^{-1} b$ and let $n = 8$.
We have $\sigma(w) = 2$ and $I(w) = [-1,3]$.
Moreover, $w$ represents the group element $(f,2)$ with 
$f(-1) = 1$, $f(0) = 2$, $f(1) = 3$, $f(2) = 4$, and $f(3) = 5$.

Let us now consider the word $w^8$. The following diagram shows how to obtain the corresponding
element of $\Z \wr \Z$:
\begin{center}
\ifFullVersion\else\small\fi
\begin{tabular}{|c|c|c|c|c|c|c|c|c|c|c|c|c|c|c|c|c|c|c|}
-1 & 0  & 1  & 2  & 3 & 4 & 5 & 6 & 7 & 8 & 9 & 10 & 11 & 12 & 13 & 14 & 15 & 16 & 17  \\ \thickhline
1 & 2 & 3 & 4 & 5 &&&&&&&&&&&&&& \\
& & 1 & 2 & 3 & 4 & 5 &&&&&&&&&&&& \\
& & & & 1 & 2 & 3 & 4 & 5 &&&&&&&&&& \\
& & & & & & 1 & 2 & 3 & 4 & 5 &&&&&&&& \\
& & & & & & & & 1 & 2 & 3 & 4 & 5 &&&&&& \\
& & & & & & & & & & 1 & 2 & 3 & 4 & 5 &&&& \\
& & & & & & & & & & & & 1 & 2 & 3 & 4 & 5 && \\
& & & & & & & & & & & & & & 1 & 2 & 3 & 4 & 5   \\ \hline
1 & 2 & 4 & 6 & 9 & 6 & 9 & 6 & 9 & 6 & 9 & 6 & 9 & 6 & 9 & 6 & 8 & 4 & 5 
\end{tabular}
\end{center}
We have $I(w^8) = [-1,17]$ and $\sigma(w^8) = 8 \sigma(w) = 16$.
If $(g,16)$ is the group element represented by $w^8$, then
the function $g$ is periodic on the interval $[2,14]$ (which includes the interval
$[-1+s,17-s]$, where $s = |I(w)|=5$) 
with period $2$.
\end{example}

\ifFullVersion
\subsection{The power word problem for $G \wr \Z$}
\fi

\ifdrawFigures
\usetikzlibrary{math}

\def\s{{ 0.5, 0, -0.5, 0, 0.5, 0, -0.5}}
\def\a{{-0.2, -0.6, -0.8, -0.6, -0.3, -0.2, -0.7}}
\def\b{{0.8, 0.4, 0.2, 0.4, 0.7, 0.8, 0.3}}
\def\n{{13, 7, 22, 8, 20, 5, 11}}
\def\p{{0,\s[0]*\n[0],\s[0]*\n[0]+\s[1]*\n[1],\s[0]*\n[0]+\s[1]*\n[1]+\s[2]*\n[2],\s[0]*\n[0]+\s[1]*\n[1]+\s[2]*\n[2]+\s[3]*\n[3],
           \s[0]*\n[0]+\s[1]*\n[1]+\s[2]*\n[2]+\s[3]*\n[3]+\s[4]*\n[4],\s[0]*\n[0]+\s[1]*\n[1]+\s[2]*\n[2]+\s[3]*\n[3]+\s[4]*\n[4]+\s[5]*\n[5],
           \s[0]*\n[0]+\s[1]*\n[1]+\s[2]*\n[2]+\s[3]*\n[3]+\s[4]*\n[4]+\s[5]*\n[5]+\s[6]*\n[6]}}           
\def\d{.1}
\def\h{{\d, \d*(\n[0]+1), \d*(\n[0]+\n[1]+1), \d*(\n[0]+\n[1]+\n[2]+1), \d*(\n[0]+\n[1]+\n[2]+\n[3]+1), \d*(\n[0]+\n[1]+\n[2]+\n[3]+\n[4]+1),
\d*(\n[0]+\n[1]+\n[2]+\n[3]+\n[4]+\n[5]+1)}}
\def\m{6}

\begin{figure}
\pgfkeys{/pgf/inner sep=0em}  
  \centering
    \begin{tikzpicture} 
    \fill [green!30] (\p[3]-1,-.5) rectangle (\p[2]+1,0);     
    \foreach \x [count=\xi] in {0,...,\m} 
      {  \fill [blue!30] (\p[\x]-1,-.5) rectangle (\p[\x]+1,0); }
     \foreach \x [count=\xi] in {0,...,\m} 
     {  \pgfmathtruncatemacro{\yend}{\n[\x]-1}
        \pgfmathtruncatemacro{\z}{10*\s[\x]}
        \foreach \y in {0,...,\yend} 
        {\draw (\p[\x]+\a[\x]+\s[\x]*\y, \h[\x]+\d*\y) -- (\p[\x]+\b[\x]+\s[\x]*\y, \h[\x]+\d*\y);
         \draw [red] (\p[\x]+\s[\x]*\y, \h[\x]+\d*\y) -- (\p[\x]+\s[\x]*\y+\s[\x], \h[\x]+\d*\y);
          \draw[help lines] (\p[\x]+\s[\x]*\y, \h[\x]+\d*\y) -- (\p[\x]+\s[\x]*\y, \h[\x]+\d*\y-\d);          
          }
        \draw[help lines] (\p[\x],0) -- (\p[\x],\h[\x]-\d);
        \ifthenelse{\NOT\z=0 \OR \x=\m}{\node[label=below:{\scriptsize $p_{\xi}$}] at (\p[\x],0) {} ;}
       {\node[label=below:{\scriptsize $p_{\xi}$}] at (\p[\x],-.25) {} ;} 
      }
      \draw[help lines] (\p[\m+1], \h[\m]+\d*\n[\m]) -- (\p[\m+1],0); 
      \node[label=below:{\scriptsize $p_{8}$}] at (\p[\m]+\s[\m]*\n[\m],-.25) {} ;
      \draw[->]  (\p[3]-1,0) -- (\p[2]+1,0);
    \end{tikzpicture} 
   \caption{\label{fig-wreath-product} The situation from the proof of Proposition~\ref{prop:TC0-reduction}.
    Horizontal lines represent shifted copies of the intervals $I(u_i) = [a_i,b_i]$.
   The lengths of the read lines are the absolute values of the shifts $\sigma(u_i)$. The blue shaded regions form the set $C$. The green shaded regions form the set $B = I \protect\setminus C$.}
\end{figure}
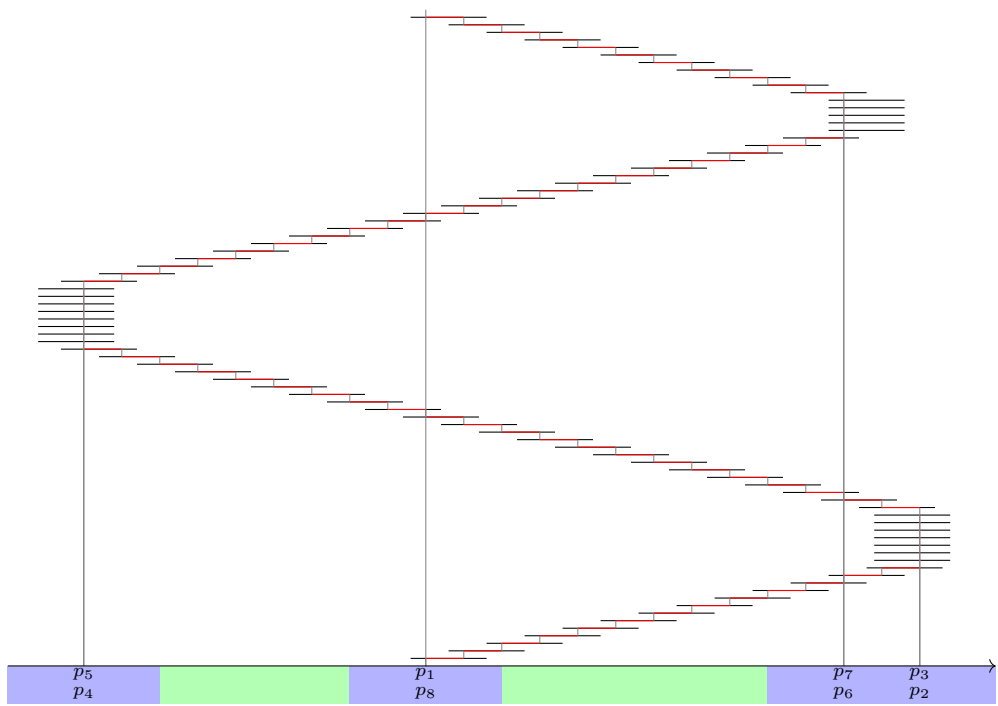
  \fi 
  
  \ifFullVersion
  \else
  \subparagraph*{Proofs of \prettyref{thm:tc0-wreath} and \ref{thm:wreath-coNP}.}
 A conjunctive truth-table reduction is a Turing reduction where the output is the conjunction over the outputs of all oracle gates; see 
 Appendix~\ref{appendix-C}.
 \fi
 
\begin{proposition} \label{prop:TC0-reduction}
For every finitely generated group $G$,  $\PowWP(G \wr \Z)$ is conjunctive truth-table \TC-reducible to $\MEM(G^\rho_{\ast})$
and $\PowWP(G)$.
\end{proposition}

\ifFullVersion
\begin{proof}
Figure~\ref{fig-wreath-product} illustrates the idea of the proof.
Let $w = u_1^{x_1} u_2^{x_2} \cdots u_k^{x_k}$
	be the input power word and
	let $(f_i,d_i) \in G \wr \Z$ be the element represented by $u_i$.
	By counting $a$'s and $a^{-1}$'s in words, we can
	compute in $\TC$ the following data:
	\begin{itemize}
		\item the numbers $d_i$ for $1 \leq i \leq k$,
		\item the (binary encodings of the endpoints of the) intervals $I(u_i) =: [a_i,b_i]$ for $1 \leq i \leq k$,
		\item the mappings $f_i$, which are represented as mappings $f_i : [a_i,b_i] \to \Sigma^*$ for $1 \leq i \leq k$,
		\item the binary encodings of the positions 
		\[
		p_i  =  \sum_{1\leq j < i} x_j d_j \ \text{ for $1 \leq i \leq k+1$,}
		\]	
		\item the intervals $I_i :=  [ \min\{ p_i, p_{i+1}-d_i\}+a_i, \max\{ p_i, p_{i+1}-d_i\}+b_i]$
	       for $1 \le i \le k$.
	\end{itemize}
	Note that $p_1 = 0$, $I_i = p_i + I(u_i^{x_i})$ and that $I_i = p_i + I(u_i)$ if $d_i = 0$.
	Moreover, we define the interval
	\[
	I = \bigcup_{1 \leq i \leq k} I_i  
	\]
	(this is indeed an interval since $p_{i+1} \in I_i$).
	Note that if $w$
	represents the group element $(f,p)$, then $p = p_{k+1}$. Hence, if $p_{k+1} \neq 0$, we can reject.
	Let us assume that $p_{k+1} = 0$ for the rest of the proof.

	Let $\ell = \max \{ b_i-a_i+1 \mid 1 \leq i \leq k \}$ be the maximal size of the intervals $I(u_i) = [a_i,b_i]$. 
	and let
	\[C := \bigcup_{i=1}^k [p_i - \ell, p_i + \ell],\] which is the union of all $\ell$-neighbourhoods of the points $p_i$.
	This set has polynomial size, and we can compute in $\TC$ a list of its elements.
	Note that $I_i \subseteq C$ if $d_i = 0$.
	
	In order to check whether $w=1$ in $G \wr \Z$, we proceed as follows:
	Let $(f,0)$ be the  group element represented by $w$.
	The support of $f$ is contained in $I$. Thus,
	it suffices to check $f(m)=1$ for all $m \in I$.
	For this, we will check whether (i) $f(m)=1$ for all $m \in C$ and (ii) $f(m)=1$ for all $m \in I \setminus C$.
	We will reduce in $\TC$ the verification of (ii) to polynomially many
	instances of $\MEM(G^\rho_{\ast})$.
	Before we do this, we first deal with (i) using
	the following claim:
	
	\medskip
	\noindent
	{\em Claim 1.} From a binary encoded  integer $m$ we can compute in $\TC$ an input instance $w_m$
	of $\PowWP(G)$ that evaluates to $1$ in $G$ if and only if $f(m)=1$ in $G$.
	
	\medskip
	\noindent
	{\em Proof of Claim 1.} We compute $w_m$ by  replacing each of the powers $u_i^{x_i}$ in our input instance $w$
	by the following rules: If $m \not\in I_i$ then we replace $u_i^{x_i}$
	by the empty word. Otherwise, we do the following:
	\begin{itemize}
		\item If $i \in A$ then we replace $u_i^{x_i}$ by $f_i(m-p_i)^{x_i}$.
		\item If $i \not\in A$, then we compute in $\TC$ the set $Q := \{ q \mid q \in [0,x_i-1], m-p_i \in q d_i + I(u_i) \}$. This is an interval
		of polynomial size. We then replace $u_i^{x_i}$ by the word
		\[
		\prod_{q \in Q} f_i(m-p_i - q d_i) \in \Sigma^* .
	\]
	\end{itemize}
	From the construction it follows that $w_m$ evaluates to the group element $f(m) \in G$.
	
	\medskip
	\noindent
	Our conjunctive truth-table \TC-reduction now outputs for every $m \in C$ the power word $w_m$.
	
	We now deal with (ii): We have to check whether $f(m)=1$ for all $m \in I \setminus C$.
	The crucial observation is that the set $B := I \setminus C$ splits into a small number of intervals,
	which can be large (at most exponential in the input length) but 
	on which $f$ is a product of periodic functions as defined in Section~\ref{sec-periodic}. 
	This allows us to reduce to the problem $\MEM(G^\rho_{\ast})$.
	We can write $B$ as a union of polynomially many disjoint intervals ($I$ is an
	interval and we remove from $I$ polynomially many points), and we can 
	compute the endpoints of these intervals $\TC$. 
	For every such interval $[b,c]$, our conjunctive truth-table \TC-reduction outputs an 
	instance of $\MEM(G^\rho_{\ast})$ that is positive if and only if $f(m)=1$ for all $m \in [b,c]$.
	The important fact is that if $[b,c]$ intersects an interval $I_i =: [a'_i,b'_i]$ then $d_i \neq 0$ 
	and $[b,c] \subseteq [a'_i+\ell,b'_i-\ell]$. Hence, if $(g_i, p_i+x_i d_i)$ is the group element represented by
	$a^{p_i} u_i^{x_i}$, then Lemma~\ref{lem:periodic-wreath} implies that the function $g_i$ is periodic on the interval $[b,c]$ with
	period $d_i$. Hence, $f$ restricted to $[b,c]$ can be obtained as the pointwise multiplication of a small number of periodic functions
	(one for each $i$ such $[b,c]$ intersects the interval $I_i$) with small period (namely, $d_i$). We can compute these periodic functions 
	(represented by non-empty words over $G$) easily in $\TC$. By shifting the interval $[b,c]$ to $[0,c-b]$ we obtain 
	the desired instance of $\MEM(G^\rho_{\ast})$.
\end{proof}

\begin{proof}[Proof of \prettyref{thm:tc0-wreath}]
	For a finitely generated abelian group, one can solve $\PowWP(G)$ in $\TC$ using the fact that
	multiplication and iterated addition on binary encoded integers can be done in $\TC$. 
	Hence, \prettyref{thm:tc0-wreath} is a consequence of \prettyref{prop:TC0-reduction} and \prettyref{thm:abelian-membership}.
\end{proof}
\else
\begin{proof}[Proof sketch]
Let $w = u_1^{x_1} u_2^{x_2} \cdots u_k^{x_k}$
be the input power word and
let $(f,d)\in G \wr \Z$
be the element represented by  $w$. We can check in $\TC$ whether $d=0$.
The difficult part is to check whether $f$ is the zero-mapping.
For this we compute an interval $I$ (of exponential size) that contains the support of $f$. 
We then partition $I$ into two sets $C$ and $I \setminus C$. The set $C$ has polynomial size 
and we can check whether $f$ is the zero-mapping on $C$ using polynomially many oracle calls to 
$\PowWP(G)$. The complement $I \setminus C$ can be written as a union of polynomially many
intervals. The crucial property of $C$ is that on each of these intervals $f$ can be written as a sum
of periodic sequences; for this we use \prettyref{lem:periodic-wreath}. Using oracle calls to $\MEM(G^\rho_{\ast})$
allows us to check whether $f$ is the zero mapping on $I \setminus C$.
\end{proof}
	Since for a finitely generated abelian group $G$, one can solve $\PowWP(G)$ in $\TC$, \prettyref{thm:tc0-wreath} is a consequence of \prettyref{prop:TC0-reduction} and \prettyref{thm:abelian-membership}.
\fi

\ifFullVersion
\section{Proof of \prettyref{thm:wreath-coNP}}\label{sec:proof-wreath-coNP}
\fi

We split the proof of \prettyref{thm:wreath-coNP} into three propositions: one for the upper bound and two for the lower bounds.
\ifFullVersion
For the upper bound we first show the following simple lemma:

\begin{lemma} \label{lem:WP->MEM}
If the word problem for the finitely generated group $G$ belongs to \coNP, then also $\MEM(G^\rho_{\ast})$ belongs to \coNP.
\end{lemma}

\begin{proof}
Assume that the word problem for $G$ belongs to \coNP. Fix a finite symmetric generating set for $G$.
Consider an input for $\MEM(G^\rho_{\ast})$, i.e., words
$u_1, \ldots, u_n \in G^+$ (elements of $G$ are represented by words from $\Sigma^*$) 
and a binary encoded number $m$.
Let $f = \bigotimes_{i=1}^n f_{u_i} : \mathbb{N} \to G$. We have to check in \coNP whether $f(k) = 1$ for all $k \in [0,m-1]$.
For this, we first guess universally a binary encoded number $k \in [0,m-1]$. Then we compute for all 
$i \in [1,n]$ the remainder $r_i = k \bmod |u_i|$ and compute the word $w_k := u_1[r_1]  \cdots u_n[r_n]$ where $u_i[r_i]$ denotes the $(r_i+1)$-st letter of $u_i$.
Since every group element in a word $u_i$ is given as a word over $\Sigma$, we can view $w_k$ as a word 
over $\Sigma$. By construction of $w_k$, it evaluates to the group element $f(k)$. 
Since $w=1$ can be checked by a \coNP-machine, we obtain a \coNP-machine for $\MEM(G^\rho_{\ast})$.
\end{proof}
\else
It is straightforward to show that if the word problem for the finitely generated group $G$ belongs to \coNP, then also $\MEM(G^\rho_{\ast})$ belongs to \coNP.
Since \coNP is closed under conjunctive truth-table \TC-reducibility,
  \prettyref{prop:TC0-reduction} yields:

\fi

\begin{proposition} \label{prop:powWP-coNP}
Let $G$ be a finitely generated group such that $\PowWP(G)$ belongs to \coNP. Then also 
$\PowWP(G \wr \mathbb{Z})$ belongs to \coNP. 
\end{proposition}
\ifFullVersion
\begin{proof}
Assume that $\PowWP(G)$ belongs to \coNP. Then also the word problem for $G$ belongs to \coNP.
By \prettyref{lem:WP->MEM}, $\MEM(G^\rho_{\ast})$ belongs to \coNP. Finally, by 
\prettyref{prop:TC0-reduction} and \prettyref{lem:conjunctive-truth-table}, $\PowWP(G \wr \mathbb{Z})$ belongs to \coNP. 
\end{proof}
\fi

\ifFullVersion
The upper bound in \prettyref{thm:wreath-coNP} is an immediate consequence 
of \prettyref{prop:powWP-coNP}, \prettyref{thm:free_power_wp} and 
\prettyref{thm:finite_index}.
\fi

\ifFullVersion
\begin{proposition}\label{prop:powWP-free-coNPhard}
Let $F$ be a finitely generated free group of rank at least two. Then $\PowWP(F \wr \mathbb{Z})$ is \coNP-hard.
\end{proposition}

\begin{proof}
	Since $F_2$ contains an isomorphic copy of $F$, it suffices to consider the wreath product $F_2 \wr \Z$.
	We prove \coNP-hardness (with respect to logspace reductions) by a reduction from the complement of the 
	satisfiability problem for boolean formulas in conjunctive normal form. Let $C = \bigwedge_{j=1}^m C_j$, where every
	$C_j$ is a clause, i.e., a disjunction of literals (possibly negated boolean variables). W.l.o.g. we can assume that $m = 2^l$
	for some $l \geq 0$.
	Let $x_1, \ldots, x_n$ be the variables appearing in $C$.
	We consider every $C_j$ as a subset of $\{x_1, \neg x_1, \ldots, x_n, \neg x_n \}$. Let $p_i$ be the $i$-th prime number for 
	$1 \leq i \leq n$; it is of order $i \cdot \ln i$. Let $M = \prod_{i=1}^n p_i$. The unary encodings of the primes $p_1, \ldots, p_n$
	and the binary encoding of the number $M$ can be computed in logspace.
	Moreover, let us define for every $1 \leq j \leq m$ (i.e., for every $C_j$) the sets
	\begin{eqnarray*}
		I^+_j &=& \{ i \mid 1 \leq i \leq n, x_i \in C_j \} \qquad \text{ and }\\
		I^-_j  &=& \{ i \mid 1 \leq i \leq n, \neg x_i \in C_j \}.
	\end{eqnarray*}
	This means $I^+_j$ are the indices of positive literals and $I^-_j$ the indices of negative literals appearing in the clause $C_j$.
	It suffices to show that the power word problem for $F_m \wr \Z$ is \coNP-hard. Note that here, $m$ is the number
	of clauses $C_j$ in $C$. The group $F_m$ (with free generators $a_1, \ldots, a_m$) can be embedded into $F_2$ (with 
	free generators $a,b$) via the morphism $a_i \mapsto a^{-i} b a^i$, and this morphism can be computed by a logspace
	transducer. For the right factor $\Z$ of the wreath product we choose the generator $b$. Hence, $\Sigma := \{ a_1, a_1^{-1},\ldots, 
	a_m, a_m^{-1}, b, b^{-1}\}$ is a symmetric generating set for the wreath product $F_m \wr \Z$.
	We first define power words $w^+(i,j)$ for all $i \in I^+_j$ and power words $w^-(i,j)$ for all $i \in I^-_j$:
	\begin{eqnarray}
	w^+(i,j) &=& (a_j b^{p_i})^{M/p_i} b^{-M}\qquad \text{ and } \label{eq-word-w+} \\
	w^-(i,j) &=& ((b a_j)^{p_i-1} b)^{M/p_i} b^{-M} . \label{eq-word-w-}
	\end{eqnarray}
	Note that the exponents $p_i$ and $p_i-1$ are of polynomial size in $n$. 
	Hence, the words $b^{p_i}$ and $(b a_j)^{p_i-1}$ can be written down explicitly by a logspace transducer.
	The exponents $M/p_i$ and $-M$ will be written down in binary notion.
	Next, for every clause $C_j$ we define the power word 
	\[
	w(C_j) = \prod_{i \in I^+_j } w^+(i,j) \prod_{i \in I^-_j } w^-(i,j) .
	\]
	The following claim is then easy to verify:
	
	\medskip
	\noindent
	{\em Claim 1.} Assume that the word $w(C_j)$ evaluates to the group element $(f_j, k_j)$ in the wreath product $F_m \wr \mathbb{Z}$.
	Then the following properties hold:
	\begin{enumerate}[(a)]
		\item \label{a} $k_j = 0$,
		\item \label{b} $f_j(z) = 1$ for all $z \in \mathbb{Z} \setminus [0,M-1]$,
		\item  \label{c} for all $z \in [0,M-1]$ we have $f_j(z) \neq 1$ if and only if 
		either there is some  $i \in I^+_j$ such that $p_i$ divides $z$ or there is some $i \in I^-_j$ such that $p_i$ does not divide $z$,
		\item \label{d} $f_j(z) \in \langle a_j \rangle$ for all $z \in \mathbb{Z}$.
	\end{enumerate}
	From \ref{c} it follows that  the following three statements are equivalent:
	\begin{enumerate}[(a)]
		\setcounter{enumi}{4}
		\item \label{e} $C$ is satisfiable.
		\item \label{f} $\exists z \in [0,M-1] \, \forall j \in [1,m] \, ( \exists i \in I^+_j : p_i$ divides $z$ or 
		$\exists i \in I^-_j : p_i$ does not divide $z$).
		\item \label{g} $\exists z \in [0,M-1] \, \forall j \in [1,m] \colon f_j(z) \neq 1$.
	\end{enumerate}
	We now reduce the latter statement to an instance of the power word problem for $F_m \wr \mathbb{N}$ using 
	a balanced binary tree of commutators. More precisely, let us define
	power words $W_{d,j}$ for $d \in [0,l]$ and $j \in [1,2^{l-d}]$ as follows (recall $m = 2^l$):
	\begin{itemize}
		\item $W_{0,j} = w(C_j)$ for $j \in [1,2^{l}] = [1,m]$,
		\item $W_{d,j} = [W_{d-1,2j-1}, W_{d-1,2j}]$ for $d \in [1,l]$ and $j \in [1,2^{l-d}]$.
	\end{itemize}
	Finally, let $W = W_{l,1}$.
	
	\medskip
	\noindent
	{\em Claim 2.} The length of $W$ is polynomially bounded in $m$ and $n$.
	Every binary encoded exponent has at most $\log M$ many bits, which is of size
	$\mathcal{O}(n \cdot \log n)$. The periods of $W$ are of length at most $2 p_n \in \mathcal{O}(n \cdot \log n)$;
	see \eqref{eq-word-w+} and  \eqref{eq-word-w-}.
	It remains to bound the number of powers $p^x$ in $W$. Every power word $w^+(i,j)$ and $w^-(i,j)$ consists of two powers.
	Hence, every power word $W_{0,j} = w(C_j)$ consists of at most $2n$ powers. By induction on $d$, it follows
	that every power word $W_{d,j}$ consists of at most $2 n 4^d$ powers. Hence, $w$ consists of 
	$2 n 4^l = 2 n m^2$ powers.
	
	\medskip
	\noindent
	{\em Claim 3.} $W=1$ in $F_m \wr \mathbb{Z}$ if and only if $C$ is not satisfiable.
	Let $(f_{d,j},k_{d,j})$ be the group element of $F_m \wr \mathbb{Z}$ represented by $W_{d,j}$; in particular,
	$f_{0,j} = f_j$.
	Consider arbitrary $d$ and $j$ with $d \in [0,l]$ and $j \in [1,2^{l-d}]$
	From the above points \ref{a} and \ref{b} it follows by induction on $d$ that $k_{d,j}=0$ and $f_{d,j}(z) = 0$ for all
	$z \in \mathbb{Z} \setminus [0,M-1]$. Moreover, \ref{d} implies that $f_{d,j}(z) \in  \langle a_{(j-1)2^d+1}, \ldots, a_{j 2^d}\rangle$.
	
	The definition of a commutator in a wreath product implies the following identity
	for all $d \in [1,l]$, $j \in [1,2^{l-d}]$ and $z \in [0,M-1]$: 
	\[
	f_{d,j}(z) =  f_{d-1,2j-1}(z) \cdot  f_{d-1,2j}(z) \cdot f_{d-1,2j-1}(z)^{-1} \cdot  f_{d-1,2j}(z)^{-1}
	\]
	(multiplication on the right hand side is in the free group $F_m$).
	Since $f_{d-1,2j-1}(z) \in  \langle a_{(j-1)2^{d}+1}, \ldots, a_{(2j-1) 2^{d-1}} \rangle$ and 
	$f_{d-1,2j}(z) \in  \langle a_{(2j-1)2^{d-1}+1}, \ldots, a_{j 2^{d}} \rangle$
	it follows that the group elements $f_{d-1,2j-1}(z)$ and $f_{d-1,2j}(z)$ commute in $F_m$ if and only if 
	$f_{d-1,2j-1}(z)=1$ or $f_{d-1,2j}(z)=1$.
	Hence, we get $f_{d,j}(z) = 1$ if and only if $f_{d-1,2j-1}(z)=1$ or $f_{d-1,2j}(z)=1$.
	For the mapping $f_{l,1}$ we thus have for all $z \in [0,M-1]$:
	$f_{l,1} = 1$ if and only if there is some $j \in [1,m]$ with  $f_j(z) = 1$.
	With the above equivalence of points \ref{e} and \ref{g} it follows that 
	$C$ is satisfiable if and only if 
	$\exists z \in [0,M-1] \, \forall j \in [1,m] \colon f_j(z) \neq 1$ 
	if and only if $\exists z \in [0,M-1] \colon f_{l,1} \neq 1$.
	Since $k_{l,1} = 0$ and $W$ represents the group element $(f_{l,1},k_{l,1})$ it follows 
	that $C$ is satisfiable if and only if $W \neq 1$ in $F_m \wr \mathbb{Z}$. This concludes
	the proof of \coNP-hardness for the case that $G$ is a finitely generated free group.
\end{proof}

\fi

\begin{proposition}\label{prop:powWP-finite-coNPhard}
	If $G$ is a finite, non-solvable group, $\PowWP(G \wr \mathbb{Z})$ is \coNP-hard.
\end{proposition}

\ifFullVersion
\begin{proof}
	Let us consider a wreath product $G \wr \mathbb{Z}$, where $G$ is a finite non-solvable
	group. Then $G$ has a subgroup $H$ such that $H = [H,H]$ (where $[H,H]$ is the commutator subgroup of $H$, i.e.,
	the subgroup generated by all commutators of $H$).
	By replacing $G$ by its subgroup $H$, we can assume that $G$ itself it equal to its commutator subgroup.
	We choose the generating set $G \setminus \{1\}$ for $G$.
	Barrington \cite{Barrington89} proved the following result:
	Let $C$ be a fan-in two boolean circuit of depth $d$ with $n$ input gates $x_1, \ldots, x_n$.
	From $C$ one can compute a sequence of triples  (a so-called $G$-program)
	\[P_{C} = (k_1,g_1,h_1) (k_2,g_2,h_2) \cdots (k_\ell,g_\ell,h_\ell) \in ([1,n] \times G \times G)^*\]
	of length $\ell \leq (4|G|)^d$ such that for every input valuation $v : \{ x_1, \ldots, x_n\} \to \{0,1\}$ the following two
	statements are equivalent:
	\begin{enumerate}[(a)]
		\item $C$ evaluates to $0$ under the input valuation $v$.
		\item $a_1 a_2 \cdots a_\ell = 1$ in $G$, where $a_i = g_i$ if $v(x_{k_i})=0$ and $a_i = h_i$ if $v(x_{k_i})=1$.
	\end{enumerate}

	Let us now take a formula $C$ in conjunctive normal form with variables $x_1, \ldots, x_n$ and $m$ clauses.
	By taking a binary tree of depth $ \mathcal{O}(\log (m+n) )$ we can write $C$ as a boolean
	circuit of depth $d \in \mathcal{O}(\log (m+n))$ with input variables $x_1, \ldots, x_n$.
	Hence, the length of the $G$-program $P_C$ is bounded by  
	$(4|G|)^d \leq (m+n)^{\mathcal{O}(1)}$ (note that $4|G|$ is a constant in our setting). 
	
	From \cite{Barrington89} it is easy to see that on input of the formula $C$ (or an arbitrary circuit of logarithmic depth), the corresponding $G$-program $P_C$ can be computed in logspace. The idea is the same as to show that $\Nc1 \sse \L$: start from the output gate and recursively evaluate the circuit storing only one bit per gate. For every gate the corresponding sequence of commutators is written on the output tape. 
	
	Let $P_C = (k_1,g_1,h_1) (k_2,g_2,h_2) \cdots (k_\ell,g_\ell,h_\ell)$.
	As in the proof for $F_2 \wr \mathbb{Z}$ we compute in logspace the 
	$n$ first primes $p_1, \ldots, p_n$ and $M = \prod_{i=1}^n p_i$ (the latter in binary notation).
	We now compute  for every $1 \leq i \leq m$ the power word (over the wreath product $G\wr \mathbb{Z}$)
	\[
	w_i = (h_i (b g_i)^{p_{k_i}-1} b)^{M/p_{k_i}} b^{-M}
	\]
	and finally compute $w_C = w_1 w_2 \cdots w_\ell$. Recall that $b$ is the generator of $\mathbb{Z}$.
	
	We claim that $w_C = 1$ in $G \wr \mathbb{Z}$ if and only if $C$ is unsatisfiable.
	Let $(f,k)$ be the group element from $G \wr \mathbb{Z}$ represented by the word $w_C$.
	We have $k=0$ and $f(z) = 1$ for all $z \in \mathbb{Z} \setminus [0,M-1]$. 
	Hence, it remains to show that $C$ is unsatisfiable
	if and only if $f(z) = 1$ for all $z \in [0,M-1]$.
	For a number $z \in [0,M-1]$ we define the valuation $v_z : \{ x_1, \ldots, x_n\} \to \{0,1\}$
	by
	\[
	v_z(x_i) = \begin{cases}
	1 &\text{if } z \equiv 0 \mod p_i\\
	0 &\text{if } z \not\equiv 0\mod p_i 
	\end{cases}
	\]
	By the Chinese remainder theorem, for every valuation $v : \{ x_1, \ldots, x_n\} \to \{0,1\}$
	there exists $z \in [0,M-1]$ with $v = v_z$.
	Moreover, from the construction of $w_C$ we get
	$f(z) = a_1 a_2 \cdots a_\ell$ where $a_i = h_i$ if $z \equiv 0 \mod  p_{k_i}$ and 
	$a_i = g_i$ if $z\not\equiv 0 \mod p_{k_i}$. In other words:
	$f(z) = a_1 a_2 \cdots a_\ell$ where $a_i = h_i$ if $v_z(x_{k_i}) = 1$ and 
	$a_i = g_i$ if $v_z(x_{k_i}) = 0$. By the equivalence of the above statements (a) and (b)
	we have $f(z) = 1$ if and only if $C$ evaluates to $0$ under the valuation $v_z$. Hence, $C$ is unsatisfiable
	if and only if $f(z) = 1$ for all $z \in [0,M-1]$.
\end{proof}

\else
\begin{proof}[Proof sketch] 
	Barrington \cite{Barrington89} proved the following result:
	Let $C$ be a fan-in two boolean circuit of depth $d$ with $n$ input gates $x_1, \ldots, x_n$.
	From $C$ one can compute a sequence of triples  (a so-called $G$-program)
	$P_{C} = (k_1,g_1,h_1) (k_2,g_2,h_2) \cdots (k_\ell,g_\ell,h_\ell) \in ([1,n] \times G \times G)^*$
	of length $\ell \leq (4|G|)^d$ such that for every input valuation $v : \{ x_1, \ldots, x_n\} \to \{0,1\}$ the following two
	statements are equivalent:
	\begin{enumerate}[(a)]
		\item $C$ evaluates to $0$ under the input valuation $v$.
		\item $a_1 a_2 \cdots a_\ell = 1$ in $G$, where $a_i = g_i$ if $v(x_{k_i})=0$ and $a_i = h_i$ if $v(x_{k_i})=1$.
	\end{enumerate}
This $G$-program is constructed as a sequence of iterated commutators, based on the observation that $[g,h] = 1$ if and only if $g=1$ or $h=1$ (given some reasonable assumptions on $g$ and $h$).
	Every formula $C$ in conjunctive normal form can be written as a circuit of depth $\mathcal{O}(\log |C|)$. Hence
	the $G$-program $p_C$ has length polynomial in $|C|$.
	From \cite{Barrington89} it is easy to see that on input of the formula $C$, the $G$-program $P_C$ can be computed in logspace.
	
	Let $P_C = (k_1,g_1,h_1) \cdots (k_\ell,g_\ell,h_\ell)$ and $x_1, \ldots, x_n$ be the variables in $C$.
	We compute in logspace the $n$ first primes $p_1, \ldots, p_n$ and $M = \prod_{i=1}^n p_i$ (the latter in binary notation).
	Let $b$ denote the generator of $\mathbb{Z}$ in the wreath product $G \wr \Z$.
	We now compute for every $1 \leq i \leq m$ the power word 
$
	w_i = (h_i (b g_i)^{p_{k_i}-1} b)^{M/p_{k_i}} b^{-M}
$
	and set $w_C = w_1 w_2 \cdots w_\ell$. 
	
	We claim that $w_C = 1$ in $G \wr \Z$ if and only if $C$ is unsatisfiable:
	For a number $z \in [0,M-1]$ we define the valuation $v_z : \{ x_1, \ldots, x_n\} \to \{0,1\}$
	by
$	v_z(x_i) = 
	1$ if $z \equiv 0 \mod p_i$ and $v_z(x_i) = 0$ otherwise.
	By the Chinese remainder theorem, for every valuation $v : \{ x_1, \ldots, x_n\} \to \{0,1\}$
	there exists $z \in [0,M-1]$ with $v = v_z$.
	Based on the above statements (a) and (b), the final step of the proof checks that 
	$f(z) = 1$ if and only if $C$ evaluates to $0$ under $v_z$.
%
%
\end{proof}

\fi

\ifFullVersion
\else
\begin{proposition}\label{prop:powWP-free-coNPhard}
	Let $F$ be a finitely generated free group of rank at least two. Then $\PowWP(F \wr \mathbb{Z})$ is \coNP-hard.
\end{proposition}
The proof is almost the same as for \prettyref{prop:powWP-finite-coNPhard}. The difference is that we mimic Robinson's proof that the word problem for $F_2$ is \Nc{1}-hard \cite{Robinson93phd} instead of Barrington's result.
\fi

\section{Further Research}
We conjecture that the method of \prettyref{sec:proof_free} can be generalized to 
right-angled Artin groups (RAAGs~-- also known as graph groups) and hyperbolic groups, and hence that the 
power word problem for a RAAG (resp., hyperbolic group) 
$G$ is \Ac{0}-Turing-reducible to the word problem for $G$.
One may also try to prove transfer results for the power word problem with respect to group theoretical
constructions, e.g., graph products, HNN extensions and amalgamated products over finite subgroups.

For finitely generated linear groups, the power word problem leads to the problem of computing matrix powers
with binary encoded exponents. The complexity of this problem is open;
variants of this problem have been studied in \cite{AllenderBD14,GalbyOW15}. 

Another open question is what happens if we allow nested exponents. 
We conjecture that in the free group for any nesting depth bounded by a constant the problem is still in 
$\AC(\WP(F_2))$. However, for unbounded nesting depth it is not clear what happens: 
we only know that it is in \P since it is a special case of the compressed word problem; but it still could be in 
$\AC(\WP(F_2))$ or it could be \P-complete or somewhere in between.



\newcommand{\Ju}{Ju}\newcommand{\Ph}{Ph}\newcommand{\Th}{Th}\newcommand{\Ch}{Ch}\newcommand{\Yu}{Yu}\newcommand{\Zh}{Zh}\newcommand{\St}{St}\newcommand{\curlybraces}[1]{\{#1\}}

\end{document}